\documentclass[a4paper,reqno,11pt]{amsart}

\usepackage[utf8]{inputenc}
\usepackage[T1]{fontenc}
\usepackage[english]{babel}
\usepackage{lmodern}
\usepackage[margin=1.0in]{geometry}
\usepackage{microtype}

\usepackage{graphicx}
\usepackage{amsmath,amssymb,amsthm}
\usepackage{mathtools}
\usepackage{bm}
\usepackage{xcolor}
\usepackage{enumitem}
\usepackage{booktabs}

\usepackage[hidelinks]{hyperref}

\allowdisplaybreaks
\numberwithin{equation}{section}

\theoremstyle{plain}

\newtheorem{thm}{Theorem}
\newtheorem{lemma}{Lemma}

\newtheorem{proposition}{Proposition}

\theoremstyle{definition}

\newtheorem{rem}{Remark}

\theoremstyle{remark}

\newcommand{\R}{{\mathbb{R}}}

\title[Drift parameter estimation for interacting particle systems with singular repulsion]{Drift parameter estimation for interacting particle systems with singular repulsion}

\author{Chiara Amorino}

\thanks{C. Amorino gratefully acknowledges financial support from
PID2022-138268NB-I00/AEI/10.13039/501100011033.}

\address{Department of Economics and Business,
Universitat Pompeu Fabra, and Barcelona School of Economics,
Ramon Trias Fargas 25-27,
08005 Barcelona, Spain}

\email{chiara.amorino@upf.edu}

\author{Jacek Ma{\l}ecki}

\thanks{Corresponding author: Jacek Ma{\l}ecki.
E-mail:
\href{mailto:jacek.malecki@pwr.edu.pl}
{jacek.malecki@pwr.edu.pl}.}

\address{Department of Mathematics,
Faculty of Pure and Applied Mathematics,
Wroc{\l}aw University of Science and Technology,
ul. Wybrze{\.z}e Wyspia{\'n}skiego 27,
50-370 Wroc{\l}aw, Poland}

\email{jacek.malecki@pwr.edu.pl}

\author{Radomyra Shevchenko}

\address{Centrale M\'editerran\'ee and
Laboratoire Jean-Alexandre Dieudonn\'e,
UMR CNRS 7351,
Universit\'e C\^ote d'Azur,
Parc Valrose,
06108 Nice Cedex 2, France}

\email{radomyra.shevchenko@univ-cotedazur.fr}



\begin{document}

\begin{abstract}
We study the estimation of the interaction parameter in singular interacting particle systems from continuous and discrete observations over a fixed time horizon. The singularity of the interaction drift prevents the direct application of standard likelihood and contrast-based methods, which rely on regularity assumptions that fail in this setting. Our approach exploits an exact reduction of the particle dynamics to a one-dimensional stochastic differential equation for a quadratic empirical moment. By the symmetry of the interaction kernel, the singular terms cancel, yielding an explicit estimator in which the unknown parameter enters linearly. We establish consistency and asymptotic normality as the number of particles, \(N\), tends to infinity. Letting \(\varepsilon_N\) denote the microscopic noise scaling of the system, we obtain asymptotic normality at the rate \(\sqrt{N/\varepsilon_N}\) for strictly positive interaction kernels, under an additional condition on the limiting empirical information. We further extend these results to degenerate interaction kernels, achieving consistency and asymptotic normality at rate \(N\) under suitable structural assumptions and the canonical scaling \(\varepsilon_N=N^{-1}\). In both regimes, we provide conditions on the observation mesh ensuring that the statistical limits remain valid under discrete sampling, in high frequency. Finally, we illustrate the applicability of our methodology through Dyson Brownian motion, the Dyson Ornstein--Uhlenbeck process, and Wishart--Laguerre particle systems, alongside a numerical investigation of the estimators' finite-sample performance.
\end{abstract}

\maketitle

\section{Introduction}
\label{s:intro}

Interacting diffusions with singular repulsion arise in several important models of random matrix theory and statistical physics. A leading example is Dyson's Brownian motion, which describes the evolution of eigenvalues of a random matrix and provides a dynamical counterpart to random matrix ensembles \cite{Dys62}. Related models include non-colliding Brownian particles, squared Bessel and Wishart--Laguerre processes, and radial Dunkl processes; see, among others, \cite{CepLep97,GraMal14,Mal26,MalPer22}. These systems describe populations of particles whose mutual repulsion becomes singular as their positions approach one another. In particular, the interaction parameter controls the strength of the repulsion and is therefore a natural quantity to estimate from observations of the particle trajectories.

Despite the extensive probabilistic literature on such models, their statistical analysis remains considerably less developed. For singular particle systems, existing statistical results have mainly focused on long-time asymptotics, in which the number of particles \(N\) is fixed and the observation horizon \(T\) tends to infinity. In this setting, ergodicity provides the mechanism underlying consistency and asymptotic normality of estimators; see \cite{AlKe12,AlKe13,DMR,MaNie, MRD} for representative results and extensions. A complementary line of research focuses on parameter estimation for regular interacting diffusions in the mean-field regime, where the observation horizon \(T\) is fixed and the number of particles tends to infinity. Starting with the seminal work of Kasonga \cite{kasonga1990maximum}, this literature has developed a variety of likelihood-based and contrast-based procedures for continuous and discrete observations, considering both finite interacting particle systems and their limiting McKean--Vlasov dynamics; see, among others, \cite{AmoHeiPilPod23,della2023lan,genon2024parametric,heidari2026local,sharrock2023online,sharrock2026efficient}. The scope of these results has since broadened considerably, with extensions to kinetic interacting particle systems \cite{amorino2026kinetic,iguchi2025parameter,iguchi2025}, systems driven by fractional Brownian motion \cite{amorino2025fractional}, and systems with jump noise \cite{duarte2026lan}. Taken together, these contributions show that observing an increasing number of particles can provide sufficient statistical information for consistent parameter estimation, even over a fixed time horizon.

A fundamental distinction between these approaches and the problem considered here lies in the regularity of the coefficients. The existing fixed-horizon methods for regular interacting diffusions rely on Lipschitz-type assumptions, together with suitable growth and non-degeneracy conditions. Such assumptions are central to the statistical arguments: they permit the control of particle increments, the approximation of the likelihood or contrast, and the derivation of the corresponding asymptotic expansions. In the presence of singular repulsion, however, the drift of each particle $X^{i,N}$ contains terms of the form
\[
\frac{H(X_t^{i,N},X_t^{j,N})}{X_t^{i,N}-X_t^{j,N}}
\]
with a symmetric function $H$, which are singular at collisions. The standard estimation procedures cannot therefore be transferred directly to the present setting: the regularity properties on which their analysis depends are no longer available. The singularity must be taken into account at the level of the statistical construction itself, rather than treated merely as an additional technical difficulty in the proof.

The aim of this paper is to develop a statistical methodology for parameter estimation in singular interacting particle systems that circumvents the regularity assumptions typically required by standard likelihood-based and contrast-based methods. The key idea is to exploit an exact reduction of the particle dynamics to a one-dimensional stochastic differential equation governing a quadratic empirical moment. This reduction removes the singularity from the estimating equation and leads to an explicit estimator of the unknown interaction parameter based on observations of the particle system. We establish the statistical properties of the proposed estimator under both continuous and discrete observation schemes over a fixed time horizon, in the asymptotic regime as the number of particles \(N\) tends to infinity.

\subsection{Motivation and main idea}

Although the statistical quantities considered below are symmetric and therefore do not depend on the labeling of the particles, the natural formulation of the singular system is in terms of ordered coordinates. Define the closed ordered Weyl chamber of type \(A_{N-1}\)
\begin{equation*}
    \overline{W}_N
    :=
    \left\{x=(x_1,\ldots,x_N)\in\mathbb{R}^N:
    x_1\leq\cdots\leq x_N\right\},
\end{equation*}
together with its collision-free part
\begin{equation*}
    W_N
    :=
    \left\{x=(x_1,\ldots,x_N)\in\mathbb{R}^N:
    x_1<\cdots<x_N\right\}.
\end{equation*}

Our basic motivating example is the scaled Dyson Brownian motion on \(\overline{W}_N\), given by
\begin{equation}
    \label{eq:intro-dyson}
    dX_t^{i,N}
    =
    \sqrt{\varepsilon_N}\,dB_t^i
    +
    \frac{\beta}{2N}
    \sum_{\substack{j=1\\j\neq i}}^N
    \frac{1}{X_t^{i,N}-X_t^{j,N}}\,dt,
    \qquad i=1,\ldots,N,
\end{equation}
where \((B^1,\ldots,B^N)\) is a collection of independent standard Brownian motions. Here, \(\beta>0\) is the unknown interaction parameter, whereas \((\varepsilon_N)_{N\geq1}\) is a known deterministic sequence describing the microscopic noise variance. Depending on the relative size of \(\varepsilon_N\) and the repulsion, the process may remain in \(W_N\) or may hit the collision boundary of the closed chamber. In the latter case, we work with a non-sticky chamber-valued solution: collisions may occur, but the process spends zero Lebesgue time on the collision set and the singular drift terms are integrable. The precise solution concept and the coefficient-level existence assumptions are discussed in detail in Section~\ref{sec:assumptions-main-results}.

Consider the quadratic empirical moment
\begin{equation}
    \label{eq:Y:defn}
    Y_N(t)
    :=
    \frac{1}{N}\sum_{i=1}^N\bigl(X_t^{i,N}\bigr)^2.
\end{equation}
At every time at which the particle coordinates are distinct, the identity
\begin{equation*}
    \sum_{i=1}^N
    \sum_{\substack{j=1\\j\neq i}}^N
    \frac{X_t^{i,N}}{X_t^{i,N}-X_t^{j,N}}
    =
    \frac{N(N-1)}{2}
\end{equation*}
holds by pairing summands at $(i,j)$ with those at $(j,i)$. In the non-sticky setting, this identity holds for Lebesgue-almost every time, which is sufficient for the integrated stochastic differential equation. Applying Itô's formula to \eqref{eq:Y:defn}, we obtain
\begin{equation}
    \label{eq:intro-dyson-Y-SDE}
    dY_N(t)
    =
    2\sqrt{\frac{\varepsilon_N}{N}}\sqrt{Y_N(t)}\,dW_t
    +
    \left(
    \varepsilon_N
    +
    \frac{\beta}{2}\left(1-\frac{1}{N}\right)
    \right)dt,
\end{equation}
where \(W\) is a one-dimensional Brownian motion, possibly after enlarging the filtered probability space (see \cite{MalPer22} for a detailed calculation in a more general case). Thus, although the original system has a singular \(N\)-dimensional drift, its quadratic empirical moment satisfies a regular one-dimensional stochastic differential equation. In particular, the singularity disappears exactly, rather than through an approximation.

The effective noise level in \eqref{eq:intro-dyson-Y-SDE} is
\begin{equation}
    \label{eq:intro-delta-r}
    \delta_N
    :=
    \sqrt{\frac{\varepsilon_N}{N}},
    \qquad
    r_N
    :=
    \delta_N^{-1}
    =
    \sqrt{\frac{N}{\varepsilon_N}}.
\end{equation}
The quantity \(r_N\) is the natural statistical rate throughout the paper. In particular, for the polynomial family \(\varepsilon_N=N^{-\alpha}\), \(\alpha\in[0,1]\), we obtain
\[
    r_N=N^{(1+\alpha)/2}.
\]
Hence the order-one microscopic noise regime \(\alpha=0\) leads to the rate \(\sqrt{N}\), whereas the normalization \(\alpha=1\) leads to the rate \(N\). The same formulation therefore covers all intermediate rates. In the degenerate interaction regime, where the repulsion coefficient may vanish at some collision points, we consider the scaling \(\varepsilon_N=N^{-1}\).

Equation~\eqref{eq:intro-dyson-Y-SDE} belongs to the small-noise diffusion framework studied, for example, in \cite{GloSor}. There is, however, an important distinction between the different scalings. When \(\varepsilon_N\equiv1\), the individual particles are not small-noise diffusions: the small noise in \eqref{eq:intro-dyson-Y-SDE} appears only after averaging over the \(N\) particles. When \(\varepsilon_N=N^{-1}\), the individual-particle noise is already of order \(N^{-1/2}\), and empirical averaging produces a further reduction in the stochastic fluctuations. The rate \eqref{eq:intro-delta-r} accommodates both mechanisms in the non-degenerate interaction regime.

Motivated by this exact reduction, we now consider the more general ordered particle system given by
\begin{equation}
    \label{eq:intro-general-system}
    dX_t^{i,N}
    =
    \sqrt{\varepsilon_N}\,
    \sigma\bigl(X_t^{i,N}\bigr)\,dB_t^i
    +
    b\bigl(X_t^{i,N}\bigr)\,dt
    +
    \frac{\beta}{N}
    \sum_{\substack{j=1\\j\neq i}}^N
    \frac{H\bigl(X_t^{i,N},X_t^{j,N}\bigr)}
    {X_t^{i,N}-X_t^{j,N}}\,dt,
    \qquad i=1,\ldots,N.
\end{equation}
Here \(b,\sigma:\mathbb{R}\to\mathbb{R}\) are known functions, and \(H:\mathbb{R}^2\to[0,\infty)\) is a known symmetric interaction kernel satisfying \(H(x,y)=H(y,x)\). The unknown parameter is again \(\beta>0\). The process
\[
X_t^N
=
\bigl(X_t^{1,N},\ldots,X_t^{N,N}\bigr)
\]
is assumed to take values in \(\overline{W}_N\). Existence, pathwise uniqueness, non-collision, and boundary non-sticking for systems of the form \eqref{eq:intro-general-system} have been studied under general assumptions on the coefficients; the precise hypotheses ensuring that the particle systems considered here are well-defined are collected in Section~\ref{sec:assumptions-main-results}. The Dyson system \eqref{eq:intro-dyson} is recovered from \eqref{eq:intro-general-system} by taking \(\sigma\equiv1\), \(b\equiv0\), and \(H(x,y)=1/2\) for every \(x,y\in\mathbb{R}\). Another important example, discussed in detail later, is the Wishart--Laguerre system, corresponding to \(\sigma(x)=\sqrt{|x|}\), \(b(x)=c>0\), and \(H(x,y)=|x|+|y|\).

As in the Dyson example, we define \(Y_N\) by \eqref{eq:Y:defn} and introduce the observable empirical processes
\begin{align*}
    A_N(t)
    &:=
    \frac{1}{N^2}
    \sum_{1\leq i<j\leq N}
    H\bigl(X_t^{i,N},X_t^{j,N}\bigr),\\
    B_N(t)
    &:=
    \frac{1}{N}
    \sum_{i=1}^N
    \left(
    2X_t^{i,N}b\bigl(X_t^{i,N}\bigr)
    +
    \varepsilon_N\sigma^2\bigl(X_t^{i,N}\bigr)
    \right),\\
    C_N(t)
    &:=
    \frac{1}{N}
    \sum_{i=1}^N
    \left(
    X_t^{i,N}\sigma\bigl(X_t^{i,N}\bigr)
    \right)^2.
\end{align*}
For every time at which the particle coordinates are distinct, the symmetry of \(H\) gives
\begin{equation*}
    \sum_{i=1}^N X_t^{i,N}
    \sum_{\substack{j=1\\j\neq i}}^N
    \frac{H\bigl(X_t^{i,N},X_t^{j,N}\bigr)}
    {X_t^{i,N}-X_t^{j,N}}
    =
    \sum_{1\leq i<j\leq N}
    H\bigl(X_t^{i,N},X_t^{j,N}\bigr).
\end{equation*}
In the non-sticky collision regime, this identity holds for Lebesgue-almost every time. It therefore yields, possibly after enlarging the filtered space, the exact semimartingale equation
\begin{equation}
    \label{eq:intro-Y-one-dimensional}
    dY_N(t)
    =
    2\sqrt{\frac{\varepsilon_N}{N}}\sqrt{C_N(t)}\,dW_t
    +
    \bigl(B_N(t)+2\beta A_N(t)\bigr)\,dt,
\end{equation}
where \(W\) is again a one-dimensional Brownian motion (see \cite{MalPer22}).

Equation~\eqref{eq:intro-Y-one-dimensional} is the key to our statistical procedure. It yields an explicit estimating equation in which the unknown parameter enters linearly, while the stochastic fluctuations are governed by the observable empirical process \(C_N\). This structure allows us to construct an estimator without requiring the regularity of the original singular drift that underlies conventional estimation procedures.

Our first set of results concerns systems with strictly positive interaction kernels. We establish consistency of the proposed estimator over a broad range of microscopic noise scalings and prove asymptotic normality at the rate \(\sqrt{N/\varepsilon_N}\) under an additional common-information condition on the reachable mean-field limits. These results hold for both continuous observations and discrete observations of the full particle configuration over the fixed time horizon $[0, T]$. In the discrete case, a vanishing observation mesh $\Delta_n$ suffices for consistency, while a stronger condition on the mesh ensures that the discretization error is negligible at the statistical rate.

We then extend the analysis to degenerate interactions, allowing the interaction kernel to vanish. Under additional structural conditions ensuring non-collision and with the canonical scaling \(\varepsilon_N=N^{-1}\), we prove consistency and asymptotic normality at the rate \(N\). We also establish the corresponding results for discrete observations, with the condition \(N\Delta_{n}\to0\) ensuring that the asymptotic Gaussian limit is preserved. Across both interaction regimes, consistency does not require all reachable mean-field limits to yield the same deterministic information, whereas the stated asymptotic normality results rely on this additional condition.

Finally, we illustrate the applicability of our theoretical results through three classical models: Dyson Brownian motion, the Dyson Ornstein--Uhlenbeck process, and Wishart--Laguerre particle systems. These examples cover different forms of singular repulsion, including both strictly positive and degenerate interaction kernels, and allow us to assess the finite-sample performance of the proposed estimators numerically.

\subsection{Outline}

The paper is organized as follows. Section~\ref{sec:assumptions-main-results} introduces the model assumptions and presents the main results on the consistency and asymptotic normality of the proposed estimators under continuous and discrete observations. Section~\ref{sec: numerical} illustrates the applicability of the general theory to Dyson Brownian motion, the Dyson Ornstein--Uhlenbeck process, and Wishart--Laguerre particle systems, and provides numerical illustrations of the finite-sample performance of the estimators. Section~\ref{s:continuous} is devoted to the proofs of the results for continuous observations. Finally, Section~\ref{s:discrete} establishes the corresponding results for discrete observations and develops the estimates needed to control the discretization error.

\section{Assumptions and main results}
\label{sec:assumptions-main-results}

\subsection{Assumptions}
We now collect the assumptions used in the statistical results. Throughout, the unknown interaction parameter satisfies $\beta>0$. Conditions~\ref{ass:epsilon}-\ref{ass:bounds} are the basic assumptions on the scaling and coefficients, while Condition~\ref{ass:finite-sigma-zero-set} is used to guarantee non-degeneracy of the observable diffusion energy. Conditions~\ref{ass:strong-regularity} and \ref{ass:noncollision-structure} are imposed only in the degenerate interaction regime, where we take $\varepsilon_N=N^{-1}$ and require the particle system to remain non-colliding.

\begin{enumerate}[label=(A\arabic*),ref=(A\arabic*)]

    \item \label{ass:epsilon}
    The sequence $(\varepsilon_N)_{N\geq2}$ is deterministic and satisfies
        \begin{equation*}
            0<\varepsilon_N\leq1, \qquad \varepsilon_N\longrightarrow\varepsilon\in[0,1].
        \end{equation*}

    \item \label{ass:continuity}
    The functions
        \begin{equation*}
            b,\sigma:\mathbb R\longrightarrow\mathbb R
        \end{equation*}
    are continuous. The interaction kernel
        \begin{equation*}
            H:\mathbb R^2\longrightarrow[0,\infty)
        \end{equation*}
    is continuous and symmetric.

    \item \label{ass:bounds}
    There exists a constant $K>0$ such that
        \begin{align*}
            \sigma^2(x)+|b(x)|^2 &\leq K(1+x^2), \qquad x\in\mathbb R, \\
            H(x,y) &\leq K(1+|x|)(1+|y|), \qquad x,y\in\mathbb R.
        \end{align*}
        
    \item \label{ass:strong-regularity}
    There exists an increasing function $\varrho:[0,\infty)\longrightarrow[0,\infty)$ such that
        \begin{equation*}
            \int_{0+}\frac{du}{\varrho(u)}=\infty
        \end{equation*}
    and
        \begin{equation*}
            |\sigma(x)-\sigma(y)|^2 \leq \varrho(|x-y|), \qquad x,y\in\mathbb R.
        \end{equation*}
    Moreover, $b$ is Lipschitz continuous. In particular, the condition on $\sigma$ is satisfied when $\sigma$ is globally $1/2$-H\"older continuous.

    \item \label{ass:noncollision-structure}
    The interaction kernel satisfies the following structural conditions:
        \begin{itemize}
            \item[(i)] for every $w<x<y<z$ we have
                \begin{equation*}
                    \frac{H(w,z)}{z-w} \leq \frac{H(x,y)}{y-x};
                \end{equation*}
            \item[(ii)] for every $x,y\in \R$ we have
                \begin{equation*}
                    \sigma^2(x)+\sigma^2(y) \leq 4H(x,y);
                \end{equation*}
            \item[(iii)] for every $x<y<z$ we have
                \begin{equation*}
                    H(x,y)(y-x)+H(y,z)(z-y) \leq H(x,z)(z-x).
                \end{equation*}
        \end{itemize}

    \item \label{ass:finite-sigma-zero-set}
    The diffusion coefficient $\sigma:\R\to\R$ has only finitely many zeros. We write
        \begin{equation*}
            M_\sigma := \#\left\{x\in\mathbb R:\sigma(x)=0\right\} <\infty.
        \end{equation*}
\end{enumerate}

Conditions~\ref{ass:strong-regularity} and \ref{ass:noncollision-structure} are not imposed in the principal strictly positive interaction regime. They will only be used in the degenerate case, where the microscopic noise is fixed at order $N^{-1/2}$ and the particles are required to remain non-colliding.

\medskip

For every $N\geq2$, let $X_0^N$ be an $\overline W_N$-valued random vector. We denote the empirical measure process by
    \begin{equation*}
        \mu_t^N := \frac1N\sum_{i=1}^N\delta_{X_t^{i,N}}, \qquad 0\leq t\leq T.
    \end{equation*}
For a probability measure $\nu$ with finite fourth moment, define
    \begin{equation*}
        \mathcal A(\nu) := \frac12 \int_{\mathbb R^2} H(x,y)\,\nu(dx)\nu(dy), \qquad \mathcal C(\nu) := \int_{\mathbb R}x^2\sigma^2(x)\,\nu(dx).
    \end{equation*}
For a measure-valued path $\nu=(\nu_t)_{t\in[0,T]}$, set
    \begin{equation*}
        \mathcal J(\nu) := \int_0^T\frac{\mathcal A(\nu_t)^2}{\mathcal C(\nu_t)} \,dt,
    \end{equation*}
whenever the right-hand side is well-defined and finite. The finite-particle quantities satisfy
    \begin{equation*}
        A_N(t) = \mathcal A(\mu_t^N) - \frac{1}{2N} \int_{\mathbb R} H(x,x)\,\mu_t^N(dx), \qquad 
        C_N(t) = \mathcal C(\mu_t^N).
    \end{equation*}
In the sequel, we will use the following assumptions related to the initial distribution of the considered particle system. 
\begin{enumerate}[label=(I\arabic*),ref=(I\arabic*)]
    
    \item \label{ass:initial-conv}
    There exists a deterministic probability measure $\mu_0\in\mathcal P(\mathbb R)$ such that
        \begin{equation*}
            \mu_0^N \xrightarrow[N\to\infty]{\mathbb P} \mu_0
        \end{equation*}
    in the topology of weak convergence.
    
    \item \label{ass:initial:moments}
    We have
    \begin{equation*}
        \sup_{N\geq2} \mathbb E\left[\int_{\mathbb R}|x|^8\,\mu_0^N(dx) \right] <\infty.
    \end{equation*}
    
    \item \label{ass:initial:non-zero}
    For every $N\geq2$ under consideration,
    \begin{equation*}
        \mathbb P\left(
            C_N(0)>0
        \right)=1.
    \end{equation*}
    Equivalently, with probability one, at least one coordinate satisfies
    \begin{equation*}
        X_0^{i,N}\sigma(X_0^{i,N})\neq0.
    \end{equation*}

    \item \label{ass:initial-limit-non-zero}
    The limiting initial diffusion energy is strictly positive $\mathcal C(\mu_0)>0$.

\end{enumerate}
Conditions~\ref{ass:initial-conv} and \ref{ass:initial:moments} are used in the mean-field compactness argument. Condition~\ref{ass:initial:non-zero} guarantees that the finite-particle estimator is well-defined from time zero, whereas Condition~\ref{ass:initial-limit-non-zero} provides the corresponding non-degeneracy at the level of the limiting empirical measure.

\medskip

For the discrete-observation results, define
    \begin{equation*}
        \phi_{B,N}(x) := 2xb(x)+\varepsilon_N\sigma^2(x), \qquad
        \phi_C(x) := x^2\sigma^2(x).
    \end{equation*}
We shall use the following additional assumption.

\begin{enumerate}[label=(D\arabic*),ref=(D\arabic*)]
    \item
    \label{ass:discrete-regularity}
    There exists a deterministic closed interval $E\subseteq\mathbb R$ with non-empty interior such that, for every particle system under consideration,
    \begin{equation*}
        \mathbb P\left(X_t^{i,N}\in E
          \text{ for all }i=1,\ldots,N\text{ and }t\in[0,T]\right)=1.
    \end{equation*}
    The restrictions of $H$ to $E^2$, and of $\phi_C$ and $\phi_{B,N}$ to $E$, admit $C^2$ extensions to open neighbourhoods of their respective domains. The extension of $H$ can be chosen symmetric. There exists a constant $K>0$, independent of $N$, such that, for $k=0,1,2$ and $x\in E$,
    \begin{align*}
        |\phi_{B,N}^{(k)}(x)|&\leq K(1+|x|^{2-k}),\\
        |\phi_C^{(k)}(x)|&\leq K(1+|x|^{4-k}).
    \end{align*}
    Moreover, for every multi-index $\gamma$ with
    $1\leq|\gamma|\leq2$,
    \begin{equation*}
        |\partial^\gamma H(x,y)|\leq K(1+x^2+y^2),
        \qquad x,y\in E.
    \end{equation*}
    Derivatives at endpoints are understood through the extensions.
\end{enumerate}
The choice $E=\mathbb R$ gives the global regularity condition. For a proper invariant interval, smoothness is required only on the part of the state space visited by the particles. The extensions used in It\^o's formula need not be non-negative outside $E^2$ and do not replace the coefficients of the original particle system.

\medskip

Before stating the main results, we recall the mean-field compactness statement that will be used throughout the paper. It follows from \cite[Theorem~8]{Mal26}, together with the general noise-scaling observation following equation~(2.25) therein. Assume Conditions~\ref{ass:epsilon}-\ref{ass:bounds} and Initial Conditions~\ref{ass:initial-conv}-\ref{ass:initial:moments}. Suppose, in addition, that for every $N$ the finite-particle system admits a continuous $\overline W_N$-valued solution on $[0,T]$ and that all singular drift terms are integrable on this interval. Then the sequence $(\mu^N)_{N\geq2}$ is tight in $C\bigl([0,T],\mathcal P(\mathbb R)\bigr)$. Moreover, the proof gives, for $p=2,4,8$,
    \begin{equation*}
        \sup_{N\geq2} \sup_{0\leq t\leq T} \mathbb E\left[\int_{\mathbb R}|x|^p\,\mu_t^N(dx)\right] <\infty.
    \end{equation*}
Every reachable subsequential limit $\mu$ satisfies, for every $f\in C_b^2(\mathbb R)$ and every $t\in[0,T]$,
    \begin{align}\label{eq:limit}
        \left\langle\mu_t,f\right\rangle &= \left\langle\mu_0,f\right\rangle + \int_0^t \left\langle\mu_s,bf'\right\rangle ds + \frac{\varepsilon}{2} \int_0^t \left\langle\mu_s,\sigma^2f''\right\rangle ds \nonumber\\
        &\quad+ \frac{\beta}{2} \int_0^t \int_{\mathbb R^2} \frac{f'(x)-f'(y)}{x-y} H(x,y)\, \mu_s(dx)\mu_s(dy) \,ds,
    \end{align}
almost surely, where the quotient in the last integral is understood as $f''(x)$ on the diagonal $x=y$, and the product $ \langle \,,\, \rangle$ denotes the integration of the second component with respect to the measure.  For a given family of particle systems, let $\mathfrak L$ denote the collection of all measure-valued processes whose laws arise as weak subsequential limits of the corresponding empirical measure processes in $C\bigl([0,T],\mathcal P(\mathbb R)\bigr)$.

\subsection{Results for strictly positive interactions}
\label{subsec:positive-results}

We first consider the principal regime in which $H$ is strictly positive. Under additional Conditions~\ref{ass:continuity} and \ref{ass:bounds}, the strictly positive existence result of \cite{Mal26} provides, for every $N$, a global weak non-sticky solution taking values in $\overline W_N$. The same result excludes positive-time collisions of three or more particles at one spatial point. By Condition~\ref{ass:finite-sigma-zero-set}, the function $x\mapsto x\sigma(x)$ has at most $M_\sigma+1$ distinct zeros. Since at most two particles can occupy any fixed point at a positive time, the condition
    \begin{equation*}
        N\geq2M_\sigma+3
    \end{equation*} 
implies $C_N(t)>0$, $t>0$. Initial Condition~\ref{ass:initial:non-zero} provides the same property at $t=0$.

\medskip

For continuous observations, set
    \begin{equation*}
        J_N := \int_0^T \frac{A_N(t)^2}{C_N(t)} \,dt
    \end{equation*}
and define
    \begin{equation*}
        \widehat{\beta}_N^{\,c} := \frac{1}{2J_N} \displaystyle\int_0^T\frac{A_N(t)}{C_N(t)}\left(dY_N(t)-B_N(t)\,dt\right).
    \end{equation*}
The following theorem gives the main result for continuous observations in the strictly positive interaction regime. It shows that the estimator $\widehat{\beta}_N^{\,c}$ is well-defined and consistent throughout the full range of noise scalings covered by Condition~\ref{ass:epsilon}. If all reachable mean-field limits have the same deterministic information, the estimator is asymptotically Gaussian at the natural rate $\sqrt{N/\varepsilon_N}$.

\begin{thm}[Continuous observations]
    \label{thm:intro-continuous}
    Assume Conditions~\ref{ass:epsilon}-\ref{ass:bounds} and \ref{ass:finite-sigma-zero-set}, and Initial Conditions~\ref{ass:initial-conv}-\ref{ass:initial-limit-non-zero}. Suppose that the repulsion coefficient is strictly positive, i.e. 
        \begin{equation}
            \label{eq:strictly-positive-H}
            H(x,y)>0\/,\qquad x,y\in\R\/.
        \end{equation}
    For every $N\geq2M_\sigma+3$, let $X^N$ be a global weak $\overline W_N$-valued solution.

    Then $\widehat{\beta}_N^{\,c}$ is well-defined almost surely for every $N\geq2M_\sigma+3$, and
        \begin{equation*}
            \widehat{\beta}_N^{\,c} \xrightarrow[N\to\infty]{\mathbb P} \beta.
        \end{equation*}
    Suppose, in addition, that there exists a deterministic constant $J_\infty\in(0,\infty)$ such that every $\mu\in\mathfrak L$ satisfies
        \begin{equation}
            \label{eq:positive-common-information}
            \mathcal J(\mu)=J_\infty \qquad\text{almost surely}.
        \end{equation}
    Then
        \begin{equation}
            \label{eq:intro-continuous-clt}
            \sqrt{\frac{N}{\varepsilon_N}} \left(\widehat{\beta}_N^{\,c}-\beta\right) \xrightarrow[N\to\infty]{\mathcal L} \mathcal N\left(0,J_\infty^{-1} \right).
        \end{equation}
\end{thm}

\bigskip

For discrete observations, let
    \begin{equation*}
        t_m=m\Delta_n, \qquad m=0,\ldots,n, \qquad \Delta_n=\frac{T}{n}.
    \end{equation*}
Define the discrete information by
    \begin{equation}
        \label{eq:discrete-information}
        J_{N,n}^{\,d} := \Delta_n \sum_{m=0}^{n-1} \frac{A_N(t_m)^2}{C_N(t_m)}
    \end{equation}
and the discrete-observation estimator by
    \begin{equation}
        \label{eq:discrete-estimator}
        \widehat{\beta}_{N,n}^{\,d} := \frac{1}{2J_{N,n}^{\,d}}\sum_{m=0}^{n-1}\frac{A_N(t_m)}{C_N(t_m)}\left[Y_N(t_{m+1})-Y_N(t_m)-B_N(t_m)\Delta_n\right]\/.
    \end{equation}
This estimator is the Euler approximation of the continuous-observation estimator associated with \eqref{eq:intro-Y-one-dimensional}. The following theorem shows that a vanishing observation mesh is sufficient for consistency. Asymptotic normality is preserved when the discretization error is negligible at the statistical rate $\sqrt{N/\varepsilon_N}$.

\begin{thm}[Discrete observations]
    \label{thm:intro-discrete}
    Assume the hypotheses of Theorem~\ref{thm:intro-continuous} preceding the additional condition \eqref{eq:positive-common-information}, and suppose that Condition~\ref{ass:discrete-regularity} holds.

    Let $(n_N)_{N\geq2M_\sigma+3}$ be a deterministic sequence such that
        \begin{equation*}
            n_N\longrightarrow\infty, \qquad \Delta_{n_N} = \frac{T}{n_N} \longrightarrow0.
        \end{equation*}
    Then $\widehat{\beta}_{N,n_N}^{\,d}$ is well-defined almost surely for every $N\geq2M_\sigma+3$, and
        \begin{equation}
            \label{eq:intro-discrete-consistency}
            \widehat{\beta}_{N,n_N}^{\,d} \xrightarrow[N\to\infty]{\mathbb P} \beta.
        \end{equation}
    If \eqref{eq:positive-common-information} also holds and
        \begin{equation}
            \label{eq:intro-discrete-mesh-condition}
            \Delta_{n_N} \sqrt{\frac{N}{\varepsilon_N}} \longrightarrow0,
        \end{equation}
    then
        \begin{equation}
            \label{eq:intro-discrete-clt}
            \sqrt{\frac{N}{\varepsilon_N}} \left(\widehat{\beta}_{N,n_N}^{\,d}-\beta\right)\xrightarrow[N\to\infty]{\mathcal L}\mathcal N\left(0,J_\infty^{-1}\right).
        \end{equation}
\end{thm}

\subsection{Degenerate interactions under the non-colliding scaling}
\label{subsec:degenerate-results}

We now allow the interaction kernel $H$ to vanish. In this regime we assume that $\beta>1$ and restrict the microscopic noise to the canonical particle scaling $\varepsilon_N= 1/N$. Conditions~\ref{ass:strong-regularity} and \ref{ass:noncollision-structure} place the particle system within the non-colliding framework of \cite{GraMal14}. Indeed, after replacing $\sigma$ by $\sigma/\sqrt N$ and $H$ by $\beta H/N$, the diffusion repulsion condition follows from
    \begin{equation*}
        \sigma^2(x)+\sigma^2(y) \leq 4H(x,y) \leq 4\beta H(x,y),
    \end{equation*}
since $\beta>1$.

Throughout this subsection we continue to work in the type-$A_{N-1}$ chamber $W_N\subset\mathbb R^N$. To avoid additional assumptions concerning instantaneous diffraction from a degenerate collision point, the finite-particle systems are started in the open chamber $W_N$. Additional support properties that arise in particular models, such as non-negativity of the particles in the Wishart-Laguerre case, will be established separately from the corresponding coe*fficients and initial conditions.

As above, $\mathfrak L$ denotes the collection of all processes which can occur as weak subsequential limits of the empirical measure processes associated with the family of particle systems considered in the present regime.

\begin{thm}[Continuous observations with degenerate interactions]
    \label{thm:degenerate-continuous}

    Assume Conditions~\ref{ass:epsilon}-\ref{ass:finite-sigma-zero-set} and Initial Conditions~\ref{ass:initial-conv}-\ref{ass:initial:moments} and~\ref{ass:initial-limit-non-zero}. Suppose that $\beta>1$, $\varepsilon_N=1/N$, and that the function $x\mapsto H(x,x)$ has at most one zero. For every $N\geq M_\sigma+2$, consider the particle system \eqref{eq:intro-general-system} started from $X_0^N\in W_N$ almost surely.

    Then the continuous-observation estimator $\widehat{\beta}_N^{\,c}$ is well-defined almost surely and
        \begin{equation}
            \label{eq:degenerate-continuous-consistency}
            \widehat{\beta}_N^{\,c} \xrightarrow[N\to\infty]{\mathbb P} \beta.
        \end{equation}
    Suppose, in addition, that there exists a deterministic constant $J_\infty\in(0,\infty)$ such that every $\mu\in\mathfrak L$ satisfies
        \begin{equation}
            \label{eq:degenerate-common-information}
            \mathcal J(\mu)=J_\infty \qquad \text{almost surely}.
        \end{equation}
    Then
        \begin{equation}
            \label{eq:degenerate-continuous-clt}
            N\left(\widehat{\beta}_N^{\,c}-\beta\right) \xrightarrow[N\to\infty]{\mathcal L}\mathcal N\left(0,J_\infty^{-1}\right).
        \end{equation}
\end{thm}

\begin{thm}[Discrete observations with degenerate interactions]
    \label{thm:degenerate-discrete}
    Assume the hypotheses of Theorem~\ref{thm:degenerate-continuous} under which \eqref{eq:degenerate-continuous-consistency} holds. Suppose, in addition, that Condition~\ref{ass:discrete-regularity} is satisfied.

    Let $(n_N)_{N\geq M_\sigma+2}$ be a deterministic sequence such that
        \begin{equation*}
            n_N\longrightarrow\infty, \qquad \Delta_{n_N} = \frac{T}{n_N} \longrightarrow0.
        \end{equation*}
    Then the discrete-observation estimator $\widehat{\beta}_{N,n_N}^{\,d}$ is well-defined almost surely and
        \begin{equation}
            \label{eq:degenerate-discrete-consistency}
            \widehat{\beta}_{N,n_N}^{\,d} \xrightarrow[N\to\infty]{\mathbb P} \beta.
        \end{equation}
    Suppose, in addition, that there exists a deterministic constant $J_\infty\in(0,\infty)$ such that every $\mu\in\mathfrak L$ satisfies
        \begin{equation*}
            \mathcal J(\mu)=J_\infty \qquad \text{almost surely},
        \end{equation*}
    and that
        \begin{equation}
            \label{eq:degenerate-discrete-mesh}
            N\Delta_{n_N}\longrightarrow0.
        \end{equation}
    Then
        \begin{equation}
            \label{eq:degenerate-discrete-clt}
            N\left(\widehat{\beta}_{N,n_N}^{\,d}-\beta\right)\xrightarrow[N\to\infty]{\mathcal L}\mathcal N\left(0,J_\infty^{-1}\right).
        \end{equation}
\end{thm}

\begin{rem}
    The assumption that $X_0^N$ belongs to the open chamber is imposed only to avoid the additional diffraction condition needed for initially colliding particles. Such initial conditions can also be treated by imposing the corresponding condition from \cite{GraMal14}, but this is more naturally verified in individual examples.

    No separate assumption that $C_N(0)>0$ is needed in this regime. Indeed, under $X_0^N\in W_N$ and $N\geq M_\sigma+2$, the initial configuration contains at least $M_\sigma+2$ distinct points, whereas $x\mapsto x\sigma(x)$ has at most $M_\sigma+1$ distinct zeros. Consequently, $C_N(0)>0$ almost surely.
\end{rem}

\section{Examples and numerical illustrations}{\label{sec: numerical}}
In this section we discuss three important examples and illustrate the performance of our estimators with numerical experiments. In all cases we use the largest microscopic noise covered by our results, namely $\varepsilon_N\equiv1$ for the two Dyson-type systems and $\varepsilon_N=1/N$ for the Laguerre system, so that the rates are $r_N=\sqrt N$ and $r_N=N$, respectively. For each configuration we generate $R$ independent replications and report bias and MSE together with the normalised mean squared error $r_N^2\,\mathrm{MSE}$. When the central limit theorems hold, $r_N^2\,\mathrm{MSE}$ should approach $J_\infty^{-1}$.

\subsection{Dyson Brownian motion}
\label{subsec:ex-dyson}

As discussed in the motivation, the Dyson Brownian motion \eqref{eq:intro-dyson} is obtained from \eqref{eq:intro-general-system} by setting $\sigma\equiv1$, $b\equiv0$ and $H\equiv\frac12$. Conditions~\ref{ass:epsilon}--\ref{ass:finite-sigma-zero-set} 
and \ref{ass:discrete-regularity} are trivially satisfied (with $M_\sigma=0$ and $E=\mathbb R$), the interaction is strictly positive, and the three observable processes simplify to
    \begin{equation*}
        A_N\equiv a_N:=\frac{N-1}{4N},\qquad B_N\equiv\varepsilon_N,\qquad C_N(t)=Y_N(t).
    \end{equation*}
The initial conditions~\ref{ass:initial-conv}--\ref{ass:initial-limit-non-zero} hold, for instance, if $X_0^N$ consists of the order statistics of $N$ i.i.d.\ random variables with law $\mu_0$ such that $\mu_0(\{0\})=0$ and $\int x^8\,\mu_0(dx)<\infty$, or if $X_0^N$ is a deterministic configuration such as the $N$ quantiles of such a law.

In this case we can compute $J_\infty$ explicitly. By the argument of \cite[Remark~1 and Section~4.1]{MalPer22}, which only uses the uniform moment bounds, the limiting equation \eqref{eq:limit} also holds for polynomial test functions $f$. Evaluating both sides for $f(x)=x^2$, for which $(f'(x)-f'(y))/(x-y)\equiv2$, yields
    \begin{equation*}
        \mathcal C(\mu_t)=\mathcal C(\mu_0)+\kappa t,\qquad \kappa:=\varepsilon+\frac\beta2,
    \end{equation*}
for every $\mu\in\mathfrak L$. Since moreover $\mathcal A(\nu)=\frac14$ for every probability measure $\nu$, condition \eqref{eq:positive-common-information} holds with
    \begin{equation}
        \label{eq:dyson-Jinf}
        J_\infty=\mathcal J(\mu)=\frac1{16}\int_0^T\frac{dt}{\mathcal C(\mu_0)+\kappa t}=\frac1{16\kappa}\log\Bigl(1+\frac{\kappa T}{\mathcal C(\mu_0)}\Bigr).
    \end{equation}
Thus Theorems~\ref{thm:intro-continuous} and~\ref{thm:intro-discrete} apply. They ensure that the estimator
    \begin{equation}
        \label{eq:dyson-cont}
        \widehat\beta^{\,c}_N=\frac{2N}{N-1}\,\frac{\displaystyle\int_0^T\frac{dY_N(t)-\varepsilon_N\,dt}{Y_N(t)}}{\displaystyle\int_0^T\frac{dt}{Y_N(t)}}
    \end{equation}
as well as its discrete-observation counterpart
    \begin{equation}
        \label{eq:dyson-disc}
        \widehat\beta^{\,d}_{N,n}=\frac{2N}{N-1}\,\frac{\displaystyle\sum_{m=0}^{n-1}\frac{Y_N(t_{m+1})-Y_N(t_m)-\varepsilon_N\Delta_n}{Y_N(t_m)}}{\displaystyle\Delta_n\sum_{m=0}^{n-1}\frac{1}{Y_N(t_m)}}
    \end{equation}
are consistent, and asymptotically normal with the explicit asymptotic variance $J_\infty^{-1}$, the latter once the mesh condition \eqref{eq:intro-discrete-mesh-condition} is satisfied.

Recall that for the Dyson Brownian motion, the reduced equation \eqref{eq:intro-dyson-Y-SDE} is an autonomous one-dimensional diffusion with small diffusion coefficient. Up to the $N$-dependence of the drift, this is the setting of \cite{GloSor}, where SDEs with small-noise asymptotics $\delta=\delta_N\to0$ are considered. In that setting, the drift parameter is estimated at rate $\delta^{-1}$, which coincides with our rate $r_N=\delta_N^{-1}$.

\textbf{Numerical results.}
Since $\widehat\beta^{\,c}_N$ depends on the particles only through $Y_N$, which is a squared Bessel process by \eqref{eq:intro-dyson-Y-SDE}, it suffices to simulate $Y_N$. We do so with the exact scheme based on its non-central chi-square transition law \cite[Section~3.4]{Glass}. We consider the discretised estimator $\widehat\beta^{\,d}_{N,n}$ with $n=1000$ and $n=10$ observations.

We take $\beta=2$, $T=1$, $\varepsilon_N\equiv1$ and the quantiles $x^0_i=-1+(2i-1)/N$ of $\mu_0=\mathcal U[-1,1]$ as initial configuration. Then $\mathcal C(\mu_0)=\frac13$ and $\kappa=2$, and \eqref{eq:dyson-Jinf} gives $J_\infty^{-1}=32/\log7\approx16.44$. We use $N\in\{250,500\}$ and $R=4000$.

The results are given in Table~\ref{tab:sim-dyson}. For $n=1000$, where $\Delta_n\sqrt N\leq0.023$, $N\,\mathrm{MSE}$ agrees with $J_\infty^{-1}$ up to the Monte Carlo error. With only $n=10$ observations, where $\Delta_n\sqrt N\approx2$, the estimator is only slightly less efficient. Since the drift of $Y_N$ is constant, its Euler approximation in \eqref{eq:dyson-disc} is exact, and the discretisation only inflates the variance.

\begin{table}[ht]
    \centering
    \small
    \begin{tabular}{rrrrr}
\toprule
$N$ & $n$ & Bias & MSE & $N\,$MSE\\
\midrule
250 & 10 & $1.18\cdot10^{-2}$ & $6.76\cdot10^{-2}$ & 16.89\\
250 & 1000 & $1.33\cdot10^{-2}$ & $6.63\cdot10^{-2}$ & 16.57\\
500 & 10 & $8.63\cdot10^{-4}$ & $3.36\cdot10^{-2}$ & 16.79\\
500 & 1000 & $1.56\cdot10^{-3}$ & $3.32\cdot10^{-2}$ & 16.62\\
\bottomrule
\end{tabular}

    \caption{Dyson Brownian motion, $\beta=2$, $\varepsilon_N\equiv1$, $T=1$, $R=4000$; estimator \eqref{eq:dyson-disc} with $n$ equidistant observations. Here $J_\infty^{-1}\approx16.44$.}
    \label{tab:sim-dyson}
\end{table}

\subsection{Dyson Ornstein--Uhlenbeck process}
\label{subsec:ex-dyson-ou}

Adding a linear restoring drift, i.e.\ taking $\sigma\equiv1$, $H\equiv\frac12$ and $b(x)=-\theta x$ with a known $\theta>0$, yields the Dyson Ornstein--Uhlenbeck process, or the Dyson process with drift. For $\varepsilon_N=1/N$ and $\beta\in\{1,2,4\}$ its stationary law is the Gaussian orthogonal, unitary and symplectic ensemble, respectively, see e.g. \cite{BCL23}. The verification of the assumptions is analogous to Section~\ref{subsec:ex-dyson}. In Condition~\ref{ass:discrete-regularity} we have $E=\mathbb R$, $\phi_{B,N}(x)=\varepsilon_N-2\theta x^2$ and $\phi_C(x)=x^2$. The observables are
    \begin{equation*}
        A_N\equiv\frac{N-1}{4N},\qquad B_N(t)=\varepsilon_N-2\theta Y_N(t),\qquad C_N(t)=Y_N(t),
    \end{equation*}
and $Y_N$ is a Cox--Ingersoll--Ross process. As in Section~\ref{subsec:ex-dyson}, evaluating the limiting equation at $f(x)=x^2$ gives $\frac{d}{dt}\mathcal C(\mu_t)=\kappa-2\theta\,\mathcal C(\mu_t)$ for every $\mu\in\mathfrak L$, i.e.
    \begin{equation*}
        \mathcal C(\mu_t)=y_*+\bigl(\mathcal C(\mu_0)-y_*\bigr)e^{-2\theta t},\qquad y_*:=\frac{\kappa}{2\theta},
    \end{equation*}
and therefore
    \begin{equation}
        \label{eq:dyson-ou-Jinf}
        J_\infty=\frac1{16}\int_0^T\frac{dt}{\mathcal C(\mu_t)}=\frac1{16y_*}\Bigl(T+\frac1{2\theta}\log\frac{\mathcal C(\mu_T)}{\mathcal C(\mu_0)}\Bigr).
    \end{equation}
The estimators take the form
    \begin{align}
        \label{eq:dyson-ou-est}
        \widehat\beta^{\,c}_N&=\frac{2N}{N-1}\,\frac{\displaystyle\int_0^T\frac{dY_N(t)-(\varepsilon_N-2\theta Y_N(t))\,dt}{Y_N(t)}}{\displaystyle\int_0^T\frac{dt}{Y_N(t)}},\\
        \label{eq:dyson-ou-est-disc}
        \widehat\beta^{\,d}_{N,n}&=\frac{2N}{N-1}\,\frac{\displaystyle\sum_{m=0}^{n-1}\frac{Y_N(t_{m+1})-Y_N(t_m)-(\varepsilon_N-2\theta Y_N(t_m))\Delta_n}{Y_N(t_m)}}{\displaystyle\Delta_n\sum_{m=0}^{n-1}\frac{1}{Y_N(t_m)}},
    \end{align}
and Theorems~\ref{thm:intro-continuous} and~\ref{thm:intro-discrete} apply as before. Since $Y_N$ is again an autonomous small-noise diffusion, the comparison with \cite{GloSor} carries over.

\textbf{Numerical results.}
We simulate $Y_N$ exactly as in Section~\ref{subsec:ex-dyson}, see \cite[Section~3.4]{Glass}. Unlike for the Dyson Brownian motion, the drift of $Y_N$ is not constant, so the Euler approximation of the drift in \eqref{eq:dyson-ou-est-disc} produces a bias of order $\Delta_n$. We use this example to illustrate the mesh condition \eqref{eq:intro-discrete-mesh-condition}, which here reads $\Delta_n\sqrt N\to0$.

We take $\beta=2$, $\theta=1$, $T=1$, $\varepsilon_N\equiv1$ and the quantiles of $\mu_0=\mathcal U[-3,3]$ as initial configuration. Then $\mathcal C(\mu_0)=3$ is far from the stationary value $y_*=1$, so that the drift of $Y_N$ is large, and \eqref{eq:dyson-ou-Jinf} gives $J_\infty^{-1}\approx28.05$. We use $N\in\{250,500\}$, $R=4000$, and $n\in\{10,100,1000\}$ observations, all taken from the same simulated paths.

The results are given in Table~\ref{tab:sim-dyson-ou}. With $n=1000$, $N\,\mathrm{MSE}$ is close to $J_\infty^{-1}$ for both values of $N$. With $n=10$, the bias is about $0.32$ for both values of $N$, i.e.\ it is of order $\Delta_n$ and does not decrease with $N$. Moreover, the normalised MSE grows from $53.1$ for $N=250$ to $76.4$ for $N=500$. With $n=100$, the bias is already small, and $N\,\mathrm{MSE}$ is only slightly above $J_\infty^{-1}$. This is in accordance with \eqref{eq:intro-discrete-mesh-condition}.

\begin{table}[ht]
    \centering
    \small
    \begin{tabular}{rrrrr}
\toprule
$N$ & $n$ & Bias & MSE & $N\,$MSE\\
\midrule
250 & 10 & $3.26\cdot10^{-1}$ & $2.13\cdot10^{-1}$ & 53.14\\
250 & 100 & $4.76\cdot10^{-2}$ & $1.15\cdot10^{-1}$ & 28.72\\
250 & 1000 & $2.06\cdot10^{-2}$ & $1.14\cdot10^{-1}$ & 28.41\\
500 & 10 & $3.16\cdot10^{-1}$ & $1.53\cdot10^{-1}$ & 76.37\\
500 & 100 & $3.70\cdot10^{-2}$ & $5.74\cdot10^{-2}$ & 28.72\\
500 & 1000 & $9.88\cdot10^{-3}$ & $5.65\cdot10^{-2}$ & 28.23\\
\bottomrule
\end{tabular}

    \caption{Dyson Ornstein--Uhlenbeck process, $\beta=2$, $\theta=1$, $\varepsilon_N\equiv1$, $T=1$, $R=4000$; estimator \eqref{eq:dyson-ou-est-disc} with $n$ equidistant observations. Here $J_\infty^{-1}\approx28.05$.}
    \label{tab:sim-dyson-ou}
\end{table}

\subsection{Laguerre (Wishart) particle systems}
\label{subsec:ex-laguerre}

Our third example has a degenerate interaction. Let $\varepsilon_N=1/N$, $\beta>1$, and
    \begin{equation*}
        \sigma(x)=2\sqrt{|x|},\qquad b(x)\equiv c,\qquad H(x,y)=|x|+|y|,
    \end{equation*}
with a known constant $c\geq\beta$. This restriction will be motivated below. Note that, as explained in \cite{GraMal14}, for $\beta=2$ and $c=p/N$ this process defines the so-called non-colliding Squared Bessel particle system, explicitly computable via the eigenvalue process of the complex Wishart process. 


Since $H(0,0)=0$, the interaction is degenerate. Conditions~\ref{ass:epsilon}--\ref{ass:bounds} are immediate, and $\sigma$ vanishes only at the origin, so $M_\sigma=1$, while $x\mapsto H(x,x)=2|x|$ has exactly one zero. Condition~\ref{ass:strong-regularity} holds with $\varrho(u)=4u$, since $|\sqrt{|x|}-\sqrt{|y|}|^2\leq|x-y|$, and $b$ is constant. Condition~\ref{ass:noncollision-structure} is verified by a simple case distinction on the signs of the arguments.
Finally, by \cite[Corollary~6.5]{GraMal14}, applied after rescaling by $N$, the particles started from a non-negative configuration remain in $[0,\infty)$ provided that $c\geq\beta(N-1)/N$. Hence, under $c\geq \beta$, Condition~\ref{ass:discrete-regularity} holds with $E=[0,\infty)$: on $E$ we have $H(x,y)=x+y$, $\phi_{B,N}(x)=(2c+4/N)\,x$ and $\phi_C(x)=4x^3$. Moreover, every $\mu_t$, $\mu\in\mathfrak L$, is supported in $[0,\infty)$.

To provide an explicit form of the estimators in this case, first let $M_N^{(k)}(t)=\frac1N\sum_{i=1}^N(X_t^{i,N})^k$ denote the $k$-th empirical moment, so that $Y_N=M_N^{(2)}$. Since $\sum_{i<j}(x_i+x_j)=(N-1)\sum_ix_i$, on $[0,\infty)^N$ we have
    \begin{equation*}
        A_N=\frac{N-1}{N}M_N^{(1)},\qquad B_N=\Bigl(2c+\frac{4}{N}\Bigr)M_N^{(1)},\qquad C_N=4M_N^{(3)},
    \end{equation*}
and the estimators depend only on the first three empirical moments:
    \begin{align}
        \label{eq:laguerre-cont}
        \widehat\beta_N^{\,c}&=\frac{N}{2(N-1)}\;\frac{\displaystyle\int_0^T\frac{M^{(1)}_N(t)}{M^{(3)}_N(t)}\Bigl(dY_N(t)-\Bigl(2c+\frac{4}{N}\Bigr)M^{(1)}_N(t)\,dt\Bigr)}{\displaystyle\int_0^T\frac{M^{(1)}_N(t)^2}{M^{(3)}_N(t)}\,dt},\\
        \label{eq:laguerre-disc}
        \widehat\beta^{\,d}_{N,n}&=\frac{N}{2(N-1)}\;\frac{\displaystyle\sum_{m=0}^{n-1}\frac{M^{(1)}_N(t_m)}{M^{(3)}_N(t_m)}\Bigl(Y_N(t_{m+1})-Y_N(t_m)-\Bigl(2c+\frac{4}{N}\Bigr)M^{(1)}_N(t_m)\Delta_n\Bigr)}{\displaystyle\Delta_n\sum_{m=0}^{n-1}\frac{M^{(1)}_N(t_m)^2}{M^{(3)}_N(t_m)}}.
    \end{align}
To compute $J_\infty$, write $m_k(t)=\int x^k\mu_t(dx)$ and $a_k=m_k(0)$ for $\mu\in\mathfrak L$. On $[0,\infty)$ we have $\mathcal A(\mu_t)=m_1(t)$ and $\mathcal C(\mu_t)=s^2m_3(t)$. As in Section~\ref{subsec:ex-dyson}, we evaluate the limiting equation (with $\varepsilon=0$) at $f(x)=x^k$, $k=1,2,3$, using
    \begin{equation*}
        \frac{f'(x)-f'(y)}{x-y}H(x,y)=0,\quad2(x+y),\quad3(x+y)^2\qquad(x,y\geq0),
    \end{equation*}
respectively. This gives the triangular linear system
    \begin{equation*}
        m_1'=c,\qquad m_2'=2(c+\beta)\,m_1,\qquad m_3'=3(c+\beta)\,m_2+3\beta\,m_1^2,
    \end{equation*}
which is the case $k\leq3$ of the moment recursion in \cite[Section~6]{MalPer22}. Its solution is the same for every $\mu\in\mathfrak L$:
    \begin{align*}
        m_1(t)&=a_1+ct,\qquad m_2(t)=a_2+2(c+\beta)\Bigl(a_1t+\frac{ct^2}{2}\Bigr),\\
        m_3(t)&=a_3+3(c+\beta)\Bigl(a_2t+(c+\beta)\Bigl(a_1t^2+\frac{ct^3}{3}\Bigr)\Bigr)+\frac{\beta}{c}\Bigl((a_1+ct)^3-a_1^3\Bigr).
    \end{align*}
Hence condition \eqref{eq:degenerate-common-information} holds with
    \begin{equation}
        \label{eq:laguerre-Jinf}
        J_\infty=\frac{1}{4}\int_0^T\frac{m_1(t)^2}{m_3(t)}\,dt,
    \end{equation}
and Theorems~\ref{thm:degenerate-continuous} and~\ref{thm:degenerate-discrete} apply to particle systems started in $W_N\cap(0,\infty)^N$.

\textbf{Numerical results.}
We simulate the particle system exactly through the eigenvalues of a complex Wishart process with $p=2N$ degrees of freedom \cite{KOC01}. After the rescaling, this corresponds to $\beta=2$ and $c=4$. We use the estimator \eqref{eq:laguerre-disc} with $n=1250$ observations. We take $T=1$ and the quantiles of $\mu_0=\mathcal U[1,2]$ as initial configuration, so that $a_1=\frac32$, $a_2=\frac73$, $a_3=\frac{15}4$, and \eqref{eq:laguerre-Jinf} gives $J_\infty^{-1}\approx24.86$. We use $N\in\{125,250\}$ and $R=100$.

The results are given in Table~\ref{tab:sim-laguerre}. For $R=100$ replications, the bias is small, and $N^2\,\mathrm{MSE}$ decreases from $32.5$ for $N=125$ to $24.1$ for $N=250$, close to $J_\infty^{-1}\approx24.86$.

\begin{table}[ht]
    \centering
    \small
 \begin{tabular}{rrrrr}
\toprule
$N$ & $n$ & Bias & MSE & $N^2\,$MSE\\
\midrule
125 & 1250 & $5.24\cdot10^{-3}$ & $2.08\cdot10^{-3}$ & 32.49\\
250 & 1250 & $4.66\cdot10^{-3}$ & $3.86\cdot10^{-4}$ & 24.12\\
\bottomrule
\end{tabular}

    \caption{Laguerre system, $\beta=2$, $c=4$, $\varepsilon_N=1/N$, $T=1$, $R=100$; estimator \eqref{eq:laguerre-disc} with $n=1250$ observations. Here $J_\infty^{-1}\approx24.86$.}
    \label{tab:sim-laguerre}
\end{table}

\section{Continuous observations}
\label{s:continuous}

This section proves Theorems~\ref{thm:intro-continuous} and \ref{thm:degenerate-continuous}. The statistical argument is common to both regimes: the estimation error is a continuous local martingale divided by its quadratic variation, up to the deterministic factor $\delta_N=\sqrt{\varepsilon_N/N}$. The regimes differ only in how we establish positivity of the observable coefficients and non-degeneracy of the limiting information.

Throughout, $T>0$ is fixed. We retain the notation $Y_N,A_N,B_N,C_N,J_N,\mathfrak L$ from the introduction and write
    \begin{equation*}
        Z_\sigma:=\{x\in\mathbb R:x\sigma(x)=0\}.
    \end{equation*}
Condition~\ref{ass:finite-sigma-zero-set} implies $1\leq\#Z_\sigma\leq M_\sigma+1$. For $f\in C^2(\mathbb R)$, we use the continuous divided difference
    \begin{equation*}
        \mathcal D_f(x,y):= \begin{cases}
                                \displaystyle\frac{f'(x)-f'(y)}{x-y},&x\neq y,\\[1.2ex]
                                f''(x),&x=y.
                            \end{cases}
    \end{equation*}
The symmetrized identities below hold for Lebesgue-a.e. time in the non-sticky regime and at every time in the non-colliding regime.

\subsection{Finite-particle preliminaries}

We first check that the coefficients used in the estimator do not vanish. The existence and collision properties needed for this verification are established within the proof.

\begin{lemma}[Positivity of the observable coefficients]
    \label{lem:continuous-finite-positivity}
    Assume the coefficient and initial conditions of Theorem~\ref{thm:intro-continuous} or Theorem~\ref{thm:degenerate-continuous}, in the corresponding regime. For every $N$ allowed in that theorem, the corresponding observable processes satisfy
        \begin{equation}
            \label{eq:continuous-finite-positivity}
            \inf_{0\leq t\leq T}C_N(t)>0, \qquad \inf_{0\leq t\leq T}A_N(t)>0 \qquad\text{almost surely}.
        \end{equation}
    In particular, $0<J_N<\infty$ and $\widehat\beta_N^{\,c}$ is well-defined almost surely.
\end{lemma}

\begin{proof}
    In the strictly positive regime, \cite[Theorem~1]{Mal26} provides a weak solution up to its explosion time. Since $H>0$, integrability of the singular drift implies that every such solution is non-sticky. The argument proving \cite[Theorem~2]{Mal26} applies to every such solution and excludes positive-time triple collisions. The required compatibility of the diffusion coefficients on collision walls is automatic, since the same function $\sigma$ is used for every particle.
    
    In the degenerate regime, apply the non-collision result for an interior initial configuration from \cite[Proposition~6]{GraMal14} with
        \begin{equation*}
            \widetilde\sigma_i(x)=\frac{\sigma(x)}{\sqrt N}, \qquad \widetilde b_i(x)=b(x), \qquad \widetilde H_{ij}(x,y)=\frac{\beta}{N}H(x,y).
        \end{equation*}
    Indeed, Condition~\ref{ass:noncollision-structure}(ii) gives
        \begin{equation*}
            \widetilde\sigma_i^2(x)+\widetilde\sigma_j^2(y) =\frac{\sigma^2(x)+\sigma^2(y)}{N} \leq\frac{4H(x,y)}{N} \leq4\widetilde H_{ij}(x,y),
        \end{equation*}
    while Condition~\ref{ass:noncollision-structure}(iii) verifies the three-particle condition (A3) therein. The drift functions are identical and Lipschitz, so their ordering condition is automatic.

    Since $X_0^N\in W_N$, Proposition~6 of \cite{GraMal14} gives non-collision without the diffraction assumption (A4) therein. The existence construction in Section~5 of that paper applies with Proposition~6 in place of its Theorem~2; see also Remark~4 therein. Once non-collision is established, Condition~\ref{ass:strong-regularity} and Condition~\ref{ass:noncollision-structure}(i) give pathwise uniqueness by the argument of \cite[Theorem~4]{GraMal14}. The Yamada-Watanabe theorem therefore gives a unique strong solution up to explosion.

    In both regimes, non-explosion follows from \cite[Theorem~5]{Mal26}. Indeed, Conditions~\ref{ass:epsilon} and \ref{ass:bounds} imply
        \begin{equation*}
            \sum_{i=1}^N x_i b(x_i) +\frac{\varepsilon_N}{2}\sum_{i=1}^N\sigma^2(x_i) +\frac{\beta}{N}\sum_{i<j}H(x_i,x_j) \leq K_\beta\left(N+\sum_{i=1}^N x_i^2\right).
        \end{equation*}
    For each fixed $N$, this verifies the radial growth condition \textup{(G1)} of that theorem. Consequently, the solutions constructed above are global.
    Recall that
        \begin{equation*}
            Z_\sigma:=\{x\in\mathbb R:x\sigma(x)=0\}, \qquad \#Z_\sigma\leq M_\sigma+1.
        \end{equation*}
    Since $C_N$ is an average of non-negative terms, $C_N(t)=0$ would require every particle to belong to $Z_\sigma$. In the strictly positive regime, at most two particles can occupy any fixed point at a positive time. Consequently, $C_N(t)=0$ for some $t>0$ would imply
        \begin{equation*}
            N\leq2\#Z_\sigma\leq2M_\sigma+2,
        \end{equation*}
    contradicting $N\geq2M_\sigma+3$. At time zero, positivity follows from Condition~\ref{ass:initial:non-zero}. In the degenerate regime, all particle locations are distinct for every $t\geq0$. Thus $C_N(t)=0$ would imply
        \begin{equation*}
            N\leq\#Z_\sigma\leq M_\sigma+1,
        \end{equation*}
    contradicting $N\geq M_\sigma+2$. Hence, in either regime, $C_N(t)>0$ throughout $[0,T]$ almost surely.

    Strict positivity of $H$ immediately gives $A_N(t)>0$ in the first regime. In the degenerate regime, Condition~\ref{ass:noncollision-structure}(ii) yields
        \begin{equation*}
            A_N(t) \geq \frac{1}{4N^2}\sum_{i<j} \left(\sigma^2(X_t^{i,N})+\sigma^2(X_t^{j,N})\right) = \frac14\left(1-\frac1N\right)\left\langle\mu_t^N,\sigma^2\right\rangle >0,
        \end{equation*}
    where the last inequality follows from $C_N(t)>0$.

    Finally, $A_N$ and $C_N$ are continuous. Their strict positivity on the compact interval $[0,T]$ therefore proves \eqref{eq:continuous-finite-positivity} and gives $0<J_N<\infty$. The weight $A_N/C_N$ is continuous and pathwise bounded on $[0,T]$, so the stochastic integral defining $\widehat\beta_N^{\,c}$ is well-defined.
\end{proof}


\subsection{A martingale argument}

The following proposition isolates the statistical part of both proofs. 

\begin{proposition}[Martingale representation and inference]
    \label{prop:continuous-martingale-criterion}
    Suppose that the particle systems satisfy \eqref{eq:intro-Y-one-dimensional} and that $C_N(t)>0$ on $[0,T]$ and $0<J_N<\infty$ almost surely. Define
    \begin{equation}
        \label{eq:continuous-score}
        M_N(t):=\frac1{\sqrt N}\sum_{i=1}^N \int_0^t \frac{A_N(s)X_s^{i,N}\sigma(X_s^{i,N})}{C_N(s)} \,dB_s^i.
    \end{equation}
    Then
    \begin{equation}
        \label{eq:continuous-score-error}
        \langle M_N\rangle_T=J_N, \qquad \widehat\beta_N^{\,c}-\beta =\delta_N\frac{M_N(T)}{J_N}.
    \end{equation}
    If $r_N^2J_N\to\infty$ in probability, then $\widehat\beta_N^{\,c}\to\beta$ in probability. If, moreover, $J_N\to J_\infty$ in probability for a deterministic $J_\infty\in(0,\infty)$, then
    \begin{equation*}
        r_N(\widehat\beta_N^{\,c}-\beta)\xrightarrow[N\to\infty]{\mathcal L} \mathcal N(0,J_\infty^{-1}).
    \end{equation*}
\end{proposition}

\begin{proof}
    The processes $X^{i,N}$, $A_N$ and $C_N$ are continuous and adapted. Since $C_N(t)>0$ throughout $[0,T]$ almost surely, the integrands in \eqref{eq:continuous-score} are predictable and pathwise bounded on this interval. Consequently, $M_N$ is a well-defined continuous local martingale with $M_N(0)=0$. Independence of the driving Brownian motions gives
        \begin{equation*}
            \langle M_N\rangle_t = \frac1N\sum_{i=1}^N\int_0^t\frac{A_N(s)^2(X_s^{i,N})^2\sigma^2(X_s^{i,N})}{C_N(s)^2}\,ds = \int_0^t\frac{A_N(s)^2}{C_N(s)}\,ds.
        \end{equation*}
    In particular, $\langle M_N\rangle_T=J_N<\infty$ almost surely. Substitution of the exact equation for $Y_N$ into the estimator gives
        \begin{equation*}
            \int_0^T\frac{A_N(t)}{C_N(t)}\bigl(dY_N(t)-B_N(t)\,dt\bigr) = 2\beta J_N+2\delta_N M_N(T),
        \end{equation*}
    which proves \eqref{eq:continuous-score-error}. To prove consistency, fix $\eta,L>0$. For any $\lambda\in\mathbb R$, It\^o's formula shows that
        \begin{equation*}
            \exp\left(\lambda M_N(t) -\frac{\lambda^2}{2}\langle M_N\rangle_t\right)
        \end{equation*}
    is a non-negative local martingale starting at one. It is therefore a supermartingale, and hence
        \begin{equation*}
            \mathbb E\left[\exp\left(\lambda M_N(T)-\frac{\lambda^2}{2}J_N\right)\right]\leq1.
        \end{equation*}
    Taking $\lambda=\eta r_N$ and $\lambda=-\eta r_N$, respectively, observe that on the corresponding events
        \begin{equation*}
            \{\pm M_N(T)>\eta r_NJ_N,\ r_N^2J_N\geq L\}
        \end{equation*}
    we have
        \begin{equation*}
            \pm\eta r_NM_N(T)-\frac{\eta^2r_N^2}{2}J_N > \frac{\eta^2r_N^2}{2}J_N\geq\frac{\eta^2L}{2}.
        \end{equation*}
    Markov's inequality therefore yields
        \begin{equation*}
            \mathbb P\bigl( \pm M_N(T)>\eta r_NJ_N,\ r_N^2J_N\geq L\bigr)\leq\exp(-\eta^2L/2).
        \end{equation*}
    Since $\delta_N=r_N^{-1}$ and $J_N>0$, the representation \eqref{eq:continuous-score-error} and the union bound give
        \begin{equation}
            \label{eq:continuous-self-normalized-bound}
            \mathbb P\bigl(|\widehat\beta_N^{\,c}-\beta|>\eta\bigr) \leq \mathbb P(r_N^2J_N<L)+2\exp(-\eta^2L/2).
        \end{equation}
    The assumption $r_N^2J_N\to\infty$ in probability makes the first term tend to zero for every fixed $L$. Letting first $N\to\infty$ and then $L\to\infty$ proves consistency.

    Suppose now that $J_N\to J_\infty$ in probability for a deterministic constant $J_\infty\in(0,\infty)$. Extend $M_N$ constantly after $T$. This gives a continuous local martingale on $[0,\infty)$ with total quadratic variation $J_N<\infty$ almost surely. By the Dambis-Dubins-Schwarz theorem in its finite-quadratic-variation form \cite[Chapter~V, Theorem~1.7]{RevuzYor1999}, on a possibly enlarged probability space there exists a standard Brownian motion $\widetilde W^{\,N}$ such that
        \begin{equation*}
            M_N(t) = \widetilde W^{\,N}(\langle M_N\rangle_t), \qquad 0\leq t\leq T.
        \end{equation*}
    In particular, $M_N(T)=\widetilde W^{\,N}(J_N)$. To justify replacing $J_N$ by $J_\infty$, fix $\eta>0$ and $0<h<J_\infty$. Then
        \begin{equation*}
              \mathbb P\left(\left|\widetilde W^{\,N}(J_N)-\widetilde W^{\,N}(J_\infty) \right|>\eta \right)\leq
            \mathbb P(|J_N-J_\infty|>h) + \mathbb P\left(\sup_{|s-J_\infty|\leq h}\left|\widetilde W^{\,N}(s)-\widetilde W^{\,N}(J_\infty)\right|>\eta\right).
        \end{equation*}
    For fixed $h$, the first probability tends to zero as $N\to\infty$. Since all the processes $\widetilde W^{\,N}$ have the standard Brownian law, the second probability has the same value for every $N$ and tends to zero as $h\downarrow0$ by continuity of Brownian paths. Consequently,
        \begin{equation*}
            \widetilde W^{\,N}(J_N)
            -\widetilde W^{\,N}(J_\infty)
            \xrightarrow{\mathbb P}0.
        \end{equation*}
    This argument does not require independence between $\widetilde W^{\,N}$ and $J_N$.

    Because $J_\infty$ is deterministic, $\widetilde W^{\,N}(J_\infty)$ has distribution $\mathcal N(0,J_\infty)$ for every $N$. It follows that
        \begin{equation*}
            M_N(T)\Longrightarrow\mathcal N(0,J_\infty).
        \end{equation*}
    Finally, \eqref{eq:continuous-score-error} gives
        \begin{equation*}
            r_N(\widehat\beta_N^{\,c}-\beta) = \frac{M_N(T)}{J_N}.
        \end{equation*}
    Since $J_N\to J_\infty>0$ in probability, Slutsky's theorem yields
        \begin{equation*}
            r_N(\widehat\beta_N^{\,c}-\beta) \Longrightarrow \mathcal N(0,J_\infty^{-1}),
        \end{equation*}
    as claimed.
\end{proof}


\subsection{Mean-field control of the information}

We next establish the compactness and continuity properties used to verify the martingale criterion. 

\begin{lemma}[Moment bounds and convergence of the observables]
    \label{lem:continuous-observable-convergence}
    Assume Conditions~\ref{ass:epsilon}-\ref{ass:bounds} and Initial Conditions~\ref{ass:initial-conv}-\ref{ass:initial:moments}. Let $X^N$ be global non-sticky solutions of \eqref{eq:intro-general-system}, with integrable singular drift terms. Then, for $p=2,4,8$,
        \begin{equation}
            \label{eq:continuous-uniform-moments}
            \sup_N\mathbb E\left[\sup_{0\leq t\leq T}\int_{\mathbb R}|x|^p\,\mu_t^N(dx)\right]<\infty.
        \end{equation}
    The empirical measure processes are tight in $C([0,T],\mathcal P(\mathbb R))$. Every reachable limit $\mu$ satisfies the limiting equation from the introduction and
        \begin{equation}
            \label{eq:continuous-limit-moment}
            \sup_{0\leq t\leq T} \int_{\mathbb R}|x|^8\,\mu_t(dx)<\infty \qquad\text{almost surely}.
        \end{equation}
    The paths $t\mapsto\mathcal A(\mu_t)$ and $t\mapsto\mathcal C(\mu_t)$ are continuous. Moreover, whenever $\mu^{N_k}\Rightarrow\mu$, one has the joint convergence
        \begin{equation}
            \label{eq:continuous-joint-observable-convergence}
            (\mu^{N_k},A_{N_k},C_{N_k}) \Longrightarrow \bigl(\mu,\mathcal A(\mu_\cdot),\mathcal C(\mu_\cdot)\bigr)
        \end{equation}
    in    $C([0,T],\mathcal P(\mathbb R))\times C([0,T],\mathbb R^2)$.
\end{lemma}

\begin{proof}
    Fix $p\in\{2,4,8\}$ and write
        \begin{equation*}
            m_{p,N}(t) :=\frac1N\sum_{i=1}^N|X_t^{i,N}|^p =\int_{\mathbb R}|x|^p\,\mu_t^N(dx).
        \end{equation*}
    Throughout the proof, constants may depend on $p$, $\beta$ and the coefficient bounds, and also on $T$ when indicated, but are independent of $N$ and of the localization levels. Since $p$ is even, $|x|^p=x^p$. Applying It\^o's formula to each particle coordinate gives
        \begin{align*}
            d(X_t^{i,N})^p &= p\sqrt{\varepsilon_N}(X_t^{i,N})^{p-1} \sigma(X_t^{i,N})\,dB_t^i + p(X_t^{i,N})^{p-1}b(X_t^{i,N})\,dt\\
            &\quad + \frac{p(p-1)\varepsilon_N}{2} (X_t^{i,N})^{p-2}\sigma^2(X_t^{i,N})\,dt + \frac{p\beta}{N}\sum_{j\ne i} \frac{(X_t^{i,N})^{p-1}H(X_t^{i,N},X_t^{j,N})}{X_t^{i,N}-X_t^{j,N}}\,dt.
        \end{align*}
    The singular drift is understood in the integrated sense. This application of It\^o's formula is justified by the assumed integrability of the singular drift terms. We now sum over the particles and divide by $N$. Symmetry of $H$ allows us to combine the terms indexed by $(i,j)$ and $(j,i)$. For a configuration with distinct coordinates,
        \begin{equation*}
            \sum_{i=1}^N\sum_{j\ne i} \frac{x_i^{p-1}H(x_i,x_j)}{x_i-x_j} = \sum_{i<j}H(x_i,x_j) \frac{x_i^{p-1}-x_j^{p-1}}{x_i-x_j} = \sum_{i<j}H(x_i,x_j)\sum_{k=0}^{p-2}x^{p-2-k}y^k,
        \end{equation*}
    which provides the continuous extension to the diagonal $\{x=y\}$ Consequently,
        \begin{equation*}
            m_{p,N}(t) = m_{p,N}(0)+\int_0^t d_{p,N}(s)\,ds+L_{p,N}(t),
        \end{equation*}
    where the drift is
        \begin{align*}
            d_{p,N}(s) &= \frac pN\sum_{i=1}^N (X_s^{i,N})^{p-1}b(X_s^{i,N})+ \frac{p(p-1)\varepsilon_N}{2N} \sum_{i=1}^N(X_s^{i,N})^{p-2}\sigma^2(X_s^{i,N})\\
            &\quad+ \frac{p\beta}{N^2}\sum_{i<j} H(X_s^{i,N},X_s^{j,N}) \frac{(X_s^{i,N})^{p-1}-(X_s^{j,N})^{p-1}}{X_s^{i,N}-X_s^{j,N}},
        \end{align*}
    with the preceding continuous interpretation on the diagonal, and the continuous local martingale is
        \begin{equation*}
            L_{p,N}(t) := \frac{p\sqrt{\varepsilon_N}}{N} \sum_{i=1}^N\int_0^t(X_s^{i,N})^{p-1}\sigma(X_s^{i,N})\,dB_s^i.
        \end{equation*}

    \medskip

    We first estimate the drift. The growth assumptions imply
        \begin{equation*}
            |x^{p-1}b(x)|+|x|^{p-2}\sigma^2(x)\leq K_p(1+|x|^p).
        \end{equation*}
    For the interaction term, the polynomial representation and the bound on $H$ give
        \begin{align*}
            H(x,y) \left|\frac{x^{p-1}-y^{p-1}}{x-y}\right| \leq K_p(1+|x|)(1+|y|)(1+|x|+|y|)^{p-2}\leq K_p(1+|x|^p+|y|^p).
        \end{align*}
    The same bound holds on the diagonal by continuity. Each term $|X_s^{i,N}|^p$ occurs in $N-1$ unordered pairs, so
        \begin{equation*}
            \frac1{N^2}\sum_{i<j} \bigl(1+|X_s^{i,N}|^p+|X_s^{j,N}|^p\bigr) \leq 1+ \frac1N\sum_{i=1}^N|X_s^{i,N}|^p = 1+m_{p,N}(s).
        \end{equation*}
    Using also $\varepsilon_N\leq1$, we obtain
        \begin{equation*}
            |d_{p,N}(s)| \leq K_p(1+m_{p,N}(s)).
        \end{equation*}
    Independence of the driving Brownian motions gives
        \begin{equation*}
            \langle L_{p,N}\rangle_t = \frac{p^2\varepsilon_N}{N^2} \sum_{i=1}^N\int_0^t |X_s^{i,N}|^{2p-2} \sigma^2(X_s^{i,N})\,ds.
        \end{equation*}
    Hence, by the growth bound on $\sigma$,
        \begin{equation*}
            \frac{d}{dt}\langle L_{p,N}\rangle_t \leq \frac{K_p\varepsilon_N}{N^2} \sum_{i=1}^N(1+|X_t^{i,N}|^p)^2 \leq \frac{K_p}{N^2} \left(\sum_{i=1}^N (1+|X_t^{i,N}|^p)\right)^2 = K_p(1+m_{p,N}(t))^2.
        \end{equation*}
    To use these estimates without assuming integrability of higher moments, put
        \begin{equation*}
            V_N(t):=1+m_{p,N}(t),\qquad \tau_R:=\inf\{t\geq0:V_N(t)\geq R\}, \qquad R>1.
        \end{equation*}
    We suppress the dependence of $\tau_R$ on $p$ and $N$. Before $\tau_R$, the quadratic variation density of $L_{p,N}$ is bounded by $K_pR^2$. Thus $L_{p,N}(t\wedge\tau_R)$ is a square-integrable martingale on $[0,T]$ and has expectation zero. Notice also that $\tau_R=0$ when $V_N(0)\geq R$,  and continuity gives
        \begin{equation*}
            \sup_{t\leq T}V_N(t\wedge\tau_R) \leq \max\{R,V_N(0)\}.
        \end{equation*}
    In particular, this stopped supremum is integrable. Taking expectations in the stopped moment equation yields
        \begin{align*}
            \mathbb E V_N(t\wedge\tau_R) &\leq \mathbb E V_N(0) + K_p\int_0^t \mathbb E\bigl[\mathbf1_{\{s<\tau_R\}}V_N(s) \bigr]\,ds\\
                                         &\leq \mathbb E V_N(0) + K_p\int_0^t \mathbb E V_N(s\wedge\tau_R)\,ds.
        \end{align*}
    Initial Condition~\ref{ass:initial:moments} gives $\sup_N\mathbb E V_N(0)<\infty$, since $p\leq8$. Gronwall's lemma therefore implies
        \begin{equation*}
            \sup_N\sup_{R>1}\sup_{t\leq T} \mathbb E V_N(t\wedge\tau_R) \leq K_{p,T}.
        \end{equation*}
    To control the supremum in time, we apply the Burkholder-Davis-Gundy inequality and then Young's inequality:
        \begin{align*}
            \mathbb E\sup_{t\leq T} |L_{p,N}(t\wedge\tau_R)| &\leq K_p\mathbb E \left(\int_0^{T\wedge\tau_R}V_N(s)^2\,ds \right)^{1/2}\\
                &\leq K_p\mathbb E\left[\left(\sup_{t\leq T}V_N(t\wedge\tau_R)\right)^{1/2}\left(\int_0^T V_N(s\wedge\tau_R)\,ds\right)^{1/2} \right]\\
                &\leq \frac12\mathbb E\sup_{t\leq T} V_N(t\wedge\tau_R) + K_p\int_0^T \mathbb E V_N(s\wedge\tau_R)\,ds.
        \end{align*}
    Combining this with the stopped moment equation and the drift estimate gives
        \begin{align*}
            \mathbb E\sup_{t\leq T}V_N(t\wedge\tau_R) &\leq \mathbb E V_N(0) + K_p\int_0^T \mathbb E V_N(s\wedge\tau_R)\,ds + \mathbb E\sup_{t\leq T} |L_{p,N}(t\wedge\tau_R)|\\
                &\leq \frac12\mathbb E\sup_{t\leq T} V_N(t\wedge\tau_R)+K_{p,T}.
        \end{align*}
    Absorbing the supremum term proves a bound independent of $N$ and $R$. Since the solutions are global and continuous, $\tau_R>T$ for all sufficiently large $R$ on each sample path. Fatou's lemma now gives
        \begin{equation*}
            \sup_N\mathbb E\left[ \sup_{t\leq T} \int_{\mathbb R}|x|^p\,\mu_t^N(dx) \right]<\infty.
        \end{equation*}
    This proves \eqref{eq:continuous-uniform-moments}. 

    \medskip

    Tightness of the empirical measure processes in $C([0,T],\mathcal P(\mathbb R))$ and the limiting equation stated in the introduction follow from \cite[Theorem~8 and the discussion following (2.25)]{Mal26}, applied with pair kernel $\beta H$ and noise coefficient $\sqrt{\varepsilon_N}\sigma$. The required coefficient, initial and solution assumptions are included in the hypotheses of the present lemma. We now establish the moment bound for a limiting path. For each fixed $t$, the map
        \begin{equation*}
            \nu\longmapsto \int_{\mathbb R}|x|^8\,\nu_t(dx)
        \end{equation*}
    is lower semicontinuous for weak convergence of measure-valued paths. Its supremum over $t\in[0,T]$ is therefore also lower semicontinuous. If $\mu^{N_k}\Rightarrow\mu$, the Portmanteau theorem and \eqref{eq:continuous-uniform-moments} give
        \begin{equation*}
            \mathbb E\left[\sup_{t\leq T}\int_{\mathbb R}|x|^8\,\mu_t(dx)\right] \leq \liminf_{k\to\infty} \mathbb E\left[\sup_{t\leq T}\int_{\mathbb R}|x|^8\,\mu_t^{N_k}(dx)\right] <\infty.
        \end{equation*}
    In particular, the supremum inside the expectation is finite almost surely, proving \eqref{eq:continuous-limit-moment}. It remains to prove convergence of the observables. Fix a subsequence with $\mu^{N_k}\Rightarrow\mu$. On a Skorokhod representation space, retaining the same notation, we may assume that
        \begin{equation*}
            \mu^{N_k}\longrightarrow\mu \quad\text{almost surely in } C([0,T],\mathcal P(\mathbb R)).
        \end{equation*}
    All the preceding moment bounds remain valid on this space because the laws are unchanged. For $R\geq1$, choose a continuous function $\chi_R$
    with values in $[0,1]$, equal to one on $[-R,R]$ and equal to zero outside $[-2R,2R]$. Define
        \begin{align*}
            \mathcal C_R(\nu) &:= \int_{\mathbb R} \chi_R(x)x^2\sigma^2(x)\,\nu(dx),\\ 
            \mathcal A_R(\nu) &:= \frac12\int_{\mathbb R^2} \chi_R(x)\chi_R(y)H(x,y)\,\nu(dx)\nu(dy).
        \end{align*}
    The truncated integrands are continuous and compactly supported. Weak convergence of the paths, together with continuity of the product-measure map, therefore gives, for each fixed $R$,
        \begin{equation*}
            \sup_{t\leq T} |\mathcal C_R(\mu_t^{N_k}) - \mathcal C_R(\mu_t)| \longrightarrow0, \qquad 
            \sup_{t\leq T} |\mathcal A_R(\mu_t^{N_k}) - \mathcal A_R(\mu_t)| \longrightarrow0
        \end{equation*}
    almost surely. In the next step, we control the omitted tails uniformly in time. The growth assumptions imply
        \begin{equation*}
            x^2\sigma^2(x)\leq K(1+|x|^4), \qquad H(x,y)\leq K(1+x^2+y^2).
        \end{equation*}
    Since $R\geq1$, it follows that
        \begin{align*}
            (1-\chi_R(x))x^2\sigma^2(x)  &\leq KR^{-4}|x|^8,\\
            (1-\chi_R(x)\chi_R(y))H(x,y) &\leq KR^{-6}(|x|^8+|y|^8).
        \end{align*}
    For the second inequality, the left-hand side vanishes unless $\max\{|x|,|y|\}>R$, and on that set
        \begin{equation*}
            1+x^2+y^2 \leq 3\max\{|x|,|y|\}^2 \leq 3R^{-6}(|x|^8+|y|^8).
        \end{equation*}
    Consequently, for every probability measure $\nu$ with finite eighth moment,
        \begin{align*}
            0\leq\mathcal C(\nu)-\mathcal C_R(\nu) &\leq KR^{-4} \int_{\mathbb R}|x|^8\,\nu(dx),\\
            0\leq\mathcal A(\nu)-\mathcal A_R(\nu) &\leq KR^{-6} \int_{\mathbb R}|x|^8\,\nu(dx).
        \end{align*}
    The moment bounds thus control the expected supremum of these errors by $K_TR^{-4}$ and $K_TR^{-6}$, respectively, both for $\mu^{N_k}$, uniformly in $k$, and for $\mu$. The triangle inequality gives
        \begin{align*}
            \sup_{t\leq T}
            |\mathcal C(\mu_t^{N_k})-\mathcal C(\mu_t)|
            & \leq
            \sup_{t\leq T}
            |\mathcal C_R(\mu_t^{N_k})
              -\mathcal C_R(\mu_t)|
            +
            \sup_{t\leq T}
            |\mathcal C(\mu_t^{N_k})
              -\mathcal C_R(\mu_t^{N_k})|\\
            &\quad +
            \sup_{t\leq T}
            |\mathcal C(\mu_t)-\mathcal C_R(\mu_t)|.
        \end{align*}
    The first term tends to zero almost surely for fixed $R$, while Markov's inequality controls the other two terms by their expected suprema. Letting first $k\to\infty$ and then $R\to\infty$, and arguing identically for $\mathcal A$, proves
        \begin{equation*}
            \sup_{t\leq T}|\mathcal C(\mu_t^{N_k})-\mathcal C(\mu_t)| \xrightarrow{\mathbb P}0, \qquad 
            \sup_{t\leq T} |\mathcal A(\mu_t^{N_k})-\mathcal A(\mu_t)| \xrightarrow{\mathbb P}0.
        \end{equation*}
    Moreover, for a fixed limiting path satisfying \eqref{eq:continuous-limit-moment}, the same tail estimates show that $\mathcal C_R(\mu_t)$ and $\mathcal A_R(\mu_t)$ converge uniformly in time as $R\to\infty$. Each truncated observable is continuous in time. Their limits $\mathcal C(\mu_t)$ and $\mathcal A(\mu_t)$ are therefore finite and continuous.

    Finally, symmetry of $H$ gives the exact identity
        \begin{equation*}
            A_N(t) = \mathcal A(\mu_t^N) - \frac1{2N}\int_{\mathbb R} H(x,x)\,\mu_t^N(dx).
        \end{equation*}
    Since $H(x,x)\leq K(1+x^2)$, we obtain
        \begin{equation*}
            \mathbb E\left[\sup_{t\leq T} |A_N(t)-\mathcal A(\mu_t^N)| \right] \leq \frac K N \left(1+\mathbb E\sup_{t\leq T}m_{2,N}(t)\right) \longrightarrow0.
        \end{equation*}
    Together with $C_N(t)=\mathcal C(\mu_t^N)$ and the preceding convergence on the representation space, this proves the joint convergence \eqref{eq:continuous-joint-observable-convergence}.
\end{proof}


\begin{lemma}[Stability of the information]
    \label{lem:continuous-information-stability}
    Assume the hypotheses of
    Lemma~\ref{lem:continuous-observable-convergence},
    and suppose that $C_N(t)>0$ throughout $[0,T]$
    and $J_N>0$ almost surely.
    Assume that every $\mu\in\mathfrak L$ satisfies
        \begin{equation}
            \label{eq:continuous-limit-positivity}
            \inf_{0\leq t\leq T}\mathcal C(\mu_t)>0,
            \qquad
            \mathcal A(\mu_t)>0
            \quad\text{for every }t\in[0,T],
        \end{equation}
    almost surely.

    Then, for every $\eta\in(0,1)$, there exists a
    deterministic constant $K_\eta\geq1$ such that,
    for all sufficiently large $N$,
        \begin{equation}
            \label{eq:continuous-information-tightness}
            \mathbb P\left(
                K_\eta^{-1}\leq J_N\leq K_\eta,\quad
                \inf_{0\leq t\leq T}C_N(t)
                \geq K_\eta^{-1}
            \right)\geq1-\eta.
        \end{equation}
    Moreover,
        \begin{equation}
            \label{eq:continuous-information-divergence}
            \frac{N}{\varepsilon_N}J_N
            \xrightarrow[N\to\infty]{\mathbb P}\infty.
        \end{equation}

    If there exists a deterministic constant
    $J_\infty\in(0,\infty)$ such that every
    $\mu\in\mathfrak L$ satisfies
    $\mathcal J(\mu)=J_\infty$ almost surely, then
        \begin{equation*}
            J_N\xrightarrow[N\to\infty]{\mathbb P}J_\infty.
        \end{equation*}
\end{lemma}

\begin{proof}
    Fix $\mu\in\mathfrak L$. By Lemma~\ref{lem:continuous-observable-convergence}, the functions $t\mapsto\mathcal A(\mu_t)$ and $t\mapsto\mathcal C(\mu_t)$ are finite and continuous. Together with \eqref{eq:continuous-limit-positivity}, this gives
        \begin{equation*}
            0<\mathcal J(\mu) =\int_0^T \frac{\mathcal A(\mu_t)^2}{\mathcal C(\mu_t)}\,dt <\infty \qquad\text{almost surely}.
        \end{equation*}
    Consider a subsequence for which $\mu^{N_k}\Rightarrow\mu$. The maps
        \begin{equation*}
            (a,c)\longmapsto \int_0^T\frac{a(t)^2}{c(t)}\,dt, \qquad c\longmapsto \left(\inf_{0\leq t\leq T}c(t)\right)^{-1}
        \end{equation*}
    are continuous in the uniform topology at every pair of continuous functions satisfying $\inf_{0\leq t\leq T}c(t)>0$. Applying the continuous mapping theorem to \eqref{eq:continuous-joint-observable-convergence} therefore yields
        \begin{equation}
            \label{eq:continuous-information-subsequence}
            J_{N_k}\Longrightarrow\mathcal J(\mu).
        \end{equation}
    Since $\mathcal J(\mu)>0$ almost surely, this convergence holds jointly with
        \begin{equation*}
            \left( J_{N_k}^{-1}, \left(\inf_{0\leq t\leq T}C_{N_k}(t)\right)^{-1}\right) \Longrightarrow \left(\mathcal J(\mu)^{-1},\left(\inf_{0\leq t\leq T}\mathcal C(\mu_t)\right)^{-1}\right).
        \end{equation*}
    All these limiting random variables are finite almost surely. By tightness of the empirical measure processes, every subsequence admits a further subsequence along which the preceding convergence applies. Consequently, the sequence of random vectors
        \begin{equation*}
            \left( J_N,\, J_N^{-1},\, \left(\inf_{0\leq t\leq T}C_N(t)\right)^{-1} \right)
        \end{equation*}
    is tight in $\mathbb R^3$. Thus, for every $\eta\in(0,1)$, there exists a deterministic $K_\eta\geq1$ such that, for all sufficiently large $N$,
        \begin{equation*}
            \mathbb P\left( J_N>K_\eta \ \text{or}\  J_N^{-1}>K_\eta \ \text{or}\  \left(\inf_{0\leq t\leq T}C_N(t)\right)^{-1} > K_\eta \right)\leq\eta.
        \end{equation*}
    Taking the complementary event gives exactly \eqref{eq:continuous-information-tightness}. We next prove \eqref{eq:continuous-information-divergence}. Fix $R>0$ and $\eta\in(0,1)$, and choose $K_\eta$ as above. Since $0<\varepsilon_N\leq1$,
        \begin{equation*}
            \frac{N}{\varepsilon_NK_\eta}\longrightarrow\infty.
        \end{equation*}
    Hence, for all sufficiently large $N$, the event $\{J_N\geq K_\eta^{-1}\}$ implies $(N/\varepsilon_N)J_N>R$. Therefore,
        \begin{equation*}
            \mathbb P\left(\frac{N}{\varepsilon_N}J_N\leq R \right) \leq \mathbb P(J_N<K_\eta^{-1}) \leq\eta.
        \end{equation*}
    Since $\eta$ is arbitrary, this proves convergence to infinity in probability.

    Finally, suppose that every reachable limit satisfies $\mathcal J(\mu)=J_\infty$ almost surely for the same deterministic constant $J_\infty\in(0,\infty)$. Every subsequence then admits a further subsequence along which \eqref{eq:continuous-information-subsequence} gives convergence in distribution to $J_\infty$. It follows that the entire sequence $J_N$ converges in distribution to this constant, and therefore also converges to it in probability.
\end{proof}


\subsection{Non-degeneracy and the continuous-observation theorems}

In the strictly positive regime, non-degeneracy of the limiting diffusion energy follows from the coefficients and the initial condition. Since $\sigma$ may have several zeros, it is necessary to exclude measures supported on the whole finite set $Z_\sigma$, not only a Dirac mass at the origin.

\begin{lemma}[Non-degeneracy under strictly positive interactions]
    \label{lem:continuous-positive-limit}
    Assume the hypotheses of Theorem~\ref{thm:intro-continuous} used for consistency. Every $\mu\in\mathfrak L$ satisfies \eqref{eq:continuous-limit-positivity} almost surely.
\end{lemma}

\begin{proof}
    Recall that
        \begin{equation*}
            Z_\sigma :=\{x\in\mathbb R:x\sigma(x)=0\} =\{0\}\cup\{x\in\mathbb R:\sigma(x)=0\}.
        \end{equation*}
    By Condition~\ref{ass:finite-sigma-zero-set}, this set is finite. Denote its distinct elements by $z_1,\ldots,z_m$, where $1\leq m\leq M_\sigma+1$. We construct a non-negative bounded test function $q$ whose zero set is exactly $Z_\sigma$. We choose $q\in C_b^2(\mathbb R)$ such that
        \begin{equation*}
            q\geq0,\qquad q^{-1}(\{0\})=Z_\sigma,\qquad q'(z_j)=0,\qquad q''(z_j)>0\quad(1\leq j\leq m),
    \end{equation*}
    and $q=1$ outside a compact interval. Such a function is obtained by composing $P(x)=\prod_{j=1}^m(x-z_j)^2$ with a smooth function on $[0,\infty)$ that equals the identity near zero, is strictly positive away from zero, and equals one for all sufficiently large arguments. In particular, $q'$ and $q''$ have compact support. Define, for $\nu\in\mathcal P(\mathbb R)$,
        \begin{equation*}
            \Phi(\nu):= \langle\nu,bq'\rangle +\frac{\varepsilon}{2}\langle\nu,\sigma^2q''\rangle +\frac{\beta}{2}\int_{\mathbb R^2} H(x,y)\mathcal D_q(x,y)\,\nu(dx)\nu(dy).
        \end{equation*}
    All integrands are bounded and continuous, so $\Phi$ is continuous for weak convergence. To check boundedness of the last integrand, note that it vanishes when both coordinates lie outside the support of $q'$, except possibly on the diagonal, where it also vanishes outside a compact interval. If only one coordinate is large, the denominator in $\mathcal D_q$ cancels its linear growth in the bound $H(x,y)\leq K(1+|x|)(1+|y|)$.

    If $\nu$ is supported on $Z_\sigma$, write $p_j=\nu(\{z_j\})$. Since $\mathcal D_q(z_i,z_j)=0$ for $i\neq j$,
    \begin{align*}
        \Phi(\nu) &=\frac{\varepsilon}{2}\sum_{j=1}^m p_j\sigma^2(z_j)q''(z_j) + \frac{\beta}{2} \sum_{j=1}^m p_j^2H(z_j,z_j)q''(z_j)\\
            &\geq \frac{\beta}{2m} \min_{1\leq j\leq m}H(z_j,z_j)q''(z_j) =:c_0>0.
    \end{align*}
    Here we used $\sum_jp_j^2\geq1/m$ and strict positivity of $H$ on the diagonal.

    We claim that there exists $\eta>0$ such that
    \begin{equation}
        \label{eq:continuous-finite-zero-barrier}
        \langle\nu,q\rangle\leq\eta \quad\Longrightarrow\quad \Phi(\nu)\geq c_0/2.
    \end{equation}
    Otherwise there would be measures $\nu_n$ with $\langle\nu_n,q\rangle\to0$ and $\Phi(\nu_n)<c_0/2$. Because $q$ is bounded away from zero outside any neighborhood of $Z_\sigma$ and equals one near infinity, this sequence is tight and any weak limit is supported on $Z_\sigma$. Continuity of $\Phi$ then contradicts its lower bound $c_0$ on such measures.

    Fix a reachable limit $\mu$ and put $u(t)=\langle\mu_t,q\rangle$. The limiting equation yields
        \begin{equation*}
            u(t)=u(0)+\int_0^t\Phi(\mu_s)\,ds.
        \end{equation*}
    Condition~\ref{ass:initial-limit-non-zero} gives $u(0)>0$, since $\mathcal C(\mu_0)>0$ means that $\mu_0$ is not supported on $Z_\sigma$. By \eqref{eq:continuous-finite-zero-barrier}, the absolutely continuous function $u$ has strictly positive derivative whenever $u\leq\eta$. A first-crossing argument therefore gives
        \begin{equation*}
            \inf_{0\leq t\leq T}u(t)\geq\min\{u(0),\eta\}>0.
        \end{equation*}
    Thus $\mu_t$ is not supported on $Z_\sigma$ at any time, and $\mathcal C(\mu_t)>0$ for every $t\in[0,T]$. Lemma~\ref{lem:continuous-observable-convergence} makes this a continuous, finite function of time, so its minimum on $[0,T]$ is strictly positive. Finally, strict positivity of $H$ gives $\mathcal A(\mu_t)>0$ for every $t$. This proves \eqref{eq:continuous-limit-positivity}.
\end{proof}

\begin{proof}[Proof of Theorem~\ref{thm:intro-continuous}]
    Lemma~\ref{lem:continuous-finite-positivity} gives a well-defined estimator for every $N\geq2M_\sigma+3$. Lemmas~\ref{lem:continuous-observable-convergence} and \ref{lem:continuous-positive-limit} verify the hypotheses of Lemma~\ref{lem:continuous-information-stability}. Hence
        \begin{equation*}
            r_N^2J_N=\frac{N}{\varepsilon_N}J_N\xrightarrow{\mathbb P}\infty.
        \end{equation*}
    Consistency follows from Proposition~\ref{prop:continuous-martingale-criterion}.

    Under \eqref{eq:positive-common-information}, Lemma~\ref{lem:continuous-information-stability} also gives $J_N\to J_\infty$ in probability. The second assertion of Proposition~\ref{prop:continuous-martingale-criterion}, with $r_N=\sqrt{N/\varepsilon_N}$, proves \eqref{eq:intro-continuous-clt}.
\end{proof}


\begin{proof}[Proof of Theorem~\ref{thm:degenerate-continuous}]
    Lemma~\ref{lem:continuous-finite-positivity} gives a well-defined estimator for every $N\geq M_\sigma+2$. Lemma~\ref{lem:continuous-observable-convergence} applies with $\varepsilon_N=N^{-1}$ and $\varepsilon=0$.

    Fix a reachable limit $\mu\in\mathfrak L$. We work on an event of probability one on which the limiting equation, the moment bounds, and continuity of the limiting observables hold. Set
        \begin{equation*}
            m(t)=\langle\mu_t,x\rangle,\qquad V(t)=\langle\mu_t,x^2\rangle-m(t)^2.
        \end{equation*}
    Testing the limiting equation with $x$ and $x^2$, justified by truncation and the moment bounds, shows that these functions are continuously differentiable. If $L$ is a Lipschitz constant for $b$, then
        \begin{align*}
            V'(t) &= \int_{\mathbb R^2} (x-y)\bigl(b(x)-b(y)\bigr)\,\mu_t(dx)\mu_t(dy)+\beta\int_{\mathbb R^2}H(x,y)\,\mu_t(dx)\mu_t(dy)\\
                &\geq -L\int_{\mathbb R^2}(x-y)^2\,\mu_t(dx)\mu_t(dy) =-2LV(t).
        \end{align*}
    Consequently,
        \begin{equation*}
            V(t)\geq e^{-2L(t-s)}V(s), \qquad 0\leq s\leq t\leq T.
        \end{equation*}
    If $V(0)>0$, this gives $V(t)>0$ for every $t>0$. If $V(0)=0$, then $\mu_0=\delta_a$ for some $a\in\mathbb R$. Initial Condition~\ref{ass:initial-limit-non-zero} gives $a^2\sigma^2(a)>0$, and Condition~\ref{ass:noncollision-structure}(ii) yields
        \begin{equation*}
            V'(0)=\beta H(a,a) \geq\frac{\beta}{2}\sigma^2(a)>0.
        \end{equation*}
    Thus $V(t)>0$ for all sufficiently small positive $t$, and the preceding inequality extends this positivity to every $t\in(0,T]$.

    Recall the finite set
        \begin{equation*}
            Z_\sigma=\{x\in\mathbb R:x\sigma(x)=0\}.
        \end{equation*}
    Choose the function $q\in C_b^2(\mathbb R)$ constructed in the proof of Lemma~\ref{lem:continuous-positive-limit}, so that
        \begin{equation*}
            q\geq0,\qquad q^{-1}(\{0\})=Z_\sigma,\qquad q'(z)=0,\qquad q''(z)>0 \quad\text{for every }z\in Z_\sigma,
        \end{equation*}
    and $q=1$ outside a compact interval. Put $u(t)=\langle\mu_t,q\rangle$. The boundedness and continuity argument in that proof applies without strict positivity of $H$. Therefore the limiting equation gives
        \begin{equation*}
            u'(t)=\langle\mu_t,bq'\rangle +\frac{\beta}{2}\int_{\mathbb R^2} H(x,y)\mathcal D_q(x,y)\,\mu_t(dx)\mu_t(dy),
        \end{equation*}
    with a continuous right-hand side. Suppose that $\mathcal C(\mu_\tau)=0$ for some $\tau\in(0,T]$. Then $\mu_\tau$ is supported on $Z_\sigma$. Writing $p_z=\mu_\tau(\{z\})$, we obtain $u(\tau)=0$ and
        \begin{equation*}
            u'(\tau)=\frac{\beta}{2} \sum_{z\in Z_\sigma}p_z^2H(z,z)q''(z)\geq0.
        \end{equation*}
    Indeed, the drift term vanishes, and $\mathcal D_q(z,w)=0$ for distinct $z,w\in Z_\sigma$. Since $u\geq0$ and $u(\tau)=0$, its derivative from the left is non-positive, also when $\tau=T$. Hence the displayed sum equals zero. Every point carrying positive mass must therefore satisfy $H(z,z)=0$. By the assumption on the diagonal of $H$, this forces $\mu_\tau$ to be a Dirac mass, contradicting $V(\tau)>0$.

    We have proved that $\mathcal C(\mu_t)>0$ for every $t\in(0,T]$. Initial Condition~\ref{ass:initial-limit-non-zero} gives positivity at time zero. Continuity from Lemma~\ref{lem:continuous-observable-convergence} now yields
        \begin{equation}
            \label{eq:degenerate-limit-nondegeneracy}
            \inf_{0\leq t\leq T}\mathcal C(\mu_t)>0 \qquad\text{almost surely}.
        \end{equation}
    Moreover, Condition~\ref{ass:noncollision-structure}(ii) gives
        \begin{equation*}
            \mathcal A(\mu_t) =\frac12\int_{\mathbb R^2} H(x,y)\,\mu_t(dx)\mu_t(dy) \geq\frac14\int_{\mathbb R} \sigma^2(x)\,\mu_t(dx)>0,
        \end{equation*}
    since $\mathcal C(\mu_t)>0$. Thus Lemma~\ref{lem:continuous-information-stability} yields
        \begin{equation*}
            N^2J_N\xrightarrow{\mathbb P}\infty.
        \end{equation*}
    Proposition~\ref{prop:continuous-martingale-criterion} proves consistency because $r_N=N$.

    If \eqref{eq:degenerate-common-information} holds, the same information-stability lemma gives $J_N\to J_\infty$ in probability. Applying the Gaussian
    conclusion of Proposition~\ref{prop:continuous-martingale-criterion} with $r_N=N$ proves \eqref{eq:degenerate-continuous-clt}.
\end{proof}


\section{Discrete observations}
\label{s:discrete}

We prove Theorems~\ref{thm:intro-discrete} and \ref{thm:degenerate-discrete} in a common framework. Throughout this section, we assume the hypotheses used for consistency in the corresponding theorem and retain the notation
    \begin{equation*}
        \delta_N=\sqrt{\frac{\varepsilon_N}{N}}, \qquad r_N=\delta_N^{-1}, \qquad t_m=m\Delta_n, \qquad \Delta_n=\frac{T}{n}.
    \end{equation*}
All limits are taken along a deterministic sequence $n=n_N$ with $\Delta_n\to0$. Unless indicated otherwise, sums over $m$ run from $0$ to $n-1$. Constants may depend on $T$, $\beta$, and the coefficient bounds, but not on $N$ or $n$. We work with the filtration $(\mathcal F_t^N)$ supporting the solution and its driving Brownian motions, including the Brownian motion $W$ in \eqref{eq:intro-Y-one-dimensional}.

\subsection{Euler contrast and error decomposition}

For discretely observed diffusion processes, the exact transition density is generally unknown. A standard alternative is to minimize a contrast function derived from a time-discretization of the process. Approximating the dynamics of $Y_N$ over each interval $[t_m,t_{m+1}]$ using the Euler-Maruyama scheme yields
    \begin{equation*}
        Y_N(t_{m+1})-Y_N(t_m) \approx \bigl(B_N(t_m)+2\beta A_N(t_m)\bigr)\Delta_n +2\delta_N\sqrt{C_N(t_m)}\,\Delta W_m,
    \end{equation*}
where $\Delta W_m=W_{t_{m+1}}-W_{t_m}$ is independent of $\mathcal F_{t_m}^N$ and has distribution $\mathcal N(0,\Delta_n)$. The contrast function penalizes the squared difference between the observed increment and the Euler drift, weighted by the inverse of the local discrete variance. We minimize
    \begin{equation*}
        U_n(\vartheta) :=\sum_{m=0}^{n-1} \frac{\left[Y_N(t_{m+1})-Y_N(t_m)-\bigl(B_N(t_m)+2\vartheta A_N(t_m)\bigr)\Delta_n\right]^2}{4\delta_N^2 C_N(t_m)\Delta_n}, \qquad \vartheta\in\mathbb R.
    \end{equation*}
By differentiating $U_n(\vartheta)$ with respect to $\vartheta$ and setting the derivative to zero, we obtain the discrete-time weighted estimator \eqref{eq:discrete-estimator} as the unique unconstrained minimizer. Indeed, the second derivative is $2r_N^2J_{N,n}^{\,d}>0$. The quantity $J_{N,n}^{\,d}$ is the normalized information term
associated with this contrast.

By Lemma~\ref{lem:continuous-finite-positivity}, $A_N$ and $C_N$ are continuous and strictly positive on $[0,T]$. Hence $0<J_{N,n}^{\,d}<\infty$ almost surely for every admissible $N$ and every $n\geq1$. Substituting the exact dynamics of $Y_N$ into \eqref{eq:discrete-estimator} gives
    \begin{align*}
        \widehat\beta_{N,n}^{\,d}-\beta &=\frac{\delta_N}{J_{N,n}^{\,d}}\sum_{m=0}^{n-1}\frac{A_N(t_m)}{C_N(t_m)}\int_{t_m}^{t_{m+1}}\sqrt{C_N(s)}\,dW_s\\
            &\quad+\frac{1}{2J_{N,n}^{\,d}} \sum_{m=0}^{n-1}\frac{A_N(t_m)}{C_N(t_m)}\int_{t_m}^{t_{m+1}}\bigl[B_N(s)-B_N(t_m)+2\beta\bigl(A_N(s)-A_N(t_m)\bigr)\bigr]ds.
    \end{align*}         
This formulation separates the error into a martingale noise term and a random time-discretization remainder. For the subsequent analysis, write
    \begin{align}
        \label{eq:def-LN-discrete}
        L_{N,n}(t) &:=\sum_{m=0}^{n-1}\frac{A_N(t_m)}{C_N(t_m)} \int_{t_m}^{t_{m+1}}\mathbf1_{\{s\leq t\}} \sqrt{C_N(s)}\,dW_s, \\
        U_{N,n} &:=\frac12\sum_{m=0}^{n-1}\frac{A_N(t_m)}{C_N(t_m)} \int_{t_m}^{t_{m+1}} \bigl[B_N(s)-B_N(t_m)+2\beta\bigl(A_N(s)-A_N(t_m)\bigr)\bigr]ds.\nonumber
\end{align}
Then
    \begin{equation}
        \label{eq:discrete-error-ratios}
        \widehat\beta_{N,n}^{\,d}-\beta=R_{N,n}+M_{N,n}, \qquad R_{N,n}:=\frac{U_{N,n}}{J_{N,n}^{\,d}}, \qquad M_{N,n}:=\delta_N\frac{L_{N,n}(T)}{J_{N,n}^{\,d}}.
    \end{equation}
Here $L_{N,n}$ is a continuous local martingale, whereas $R_{N,n}$ is a random discretization remainder. The quadratic variation of $L_{N,n}$ satisfies
    \begin{align*}
        \langle L_{N,n}\rangle_T &=\sum_{m=0}^{n-1}\frac{A_N(t_m)^2}{C_N(t_m)^2} \int_{t_m}^{t_{m+1}}C_N(s)\,ds =J_{N,n}^{\,d}+\mathcal E_{N,n}^{(C)}, \\
        \mathcal E_{N,n}^{(C)} &:=\sum_{m=0}^{n-1}\frac{A_N(t_m)^2}{C_N(t_m)^2} \int_{t_m}^{t_{m+1}}\bigl(C_N(s)-C_N(t_m)\bigr)ds.
    \end{align*}

\subsection{Macroscopic regularity and discrete information}

Our estimator depends exclusively on the macroscopic empirical processes $A_N,B_N,C_N$. By applying It\^o's formula directly to these quantities, the singular interaction drift symmetrizes into divided differences. The regularity in Condition~\ref{ass:discrete-regularity} removes the singularity and gives the following bounds. The proof is postponed to the end of this section.

\begin{lemma}[Macroscopic regularity and continuity]
    \label{l:macroscopic-regularity}
    Assume Conditions~\ref{ass:epsilon}-\ref{ass:bounds}, Initial Condition~\ref{ass:initial:moments}, and Condition~\ref{ass:discrete-regularity}. For each $N$ under consideration, let $X^N$ be a global continuous non-sticky $\overline W_N$-valued solution of \eqref{eq:intro-general-system}, with integrable singular drift terms.

    Then there exists a constant $K_T<\infty$, independent of $N$, such that
        \begin{equation}
            \label{eq:discrete-observable-moments}
            \sup_N\mathbb E\left[\sup_{t\leq T} \bigl(|A_N(t)|^4+|B_N(t)|^4+|C_N(t)|^2\bigr) \right]\leq K_T.
        \end{equation}
    For $F\in\{A,B,C\}$, there are predictable processes $D_N^F$ and square-integrable continuous martingales $M_N^F$, starting at zero, such that
        \begin{equation*}
            F_N(t)=F_N(0)+\int_0^tD_N^F(s)\,ds+M_N^F(t).
        \end{equation*}
    Their drifts and quadratic variations satisfy
        \begin{align}
            \label{eq:discrete-drift-bound}
            \sup_N\sup_{t\leq T}\mathbb E\left[|D_N^A(t)|^2+|D_N^B(t)|^2+|D_N^C(t)|^2\right]&\leq K_T, \\
            \label{eq:discrete-macroscopic-bracket-bound}
            \sup_N\sup_{t\leq T}\delta_N^{-2}\mathbb E\left[\frac{d}{dt}\bigl(\langle M_N^A\rangle_t+\langle M_N^B\rangle_t+\langle M_N^C\rangle_t\bigr)\right]&\leq K_T.
        \end{align}
    Consequently, for $0\leq s<t\leq T$,
        \begin{equation}
            \label{eq:discrete-macroscopic-increments}
            \mathbb E\left[|A_N(t)-A_N(s)|^2+|B_N(t)-B_N(s)|^2+|C_N(t)-C_N(s)|^2\right] \leq K_T\bigl((t-s)^2+\delta_N^2(t-s)\bigr).
        \end{equation}
\end{lemma}

The continuous-observation results provide the required control of the information and the denominators. By Lemma~\ref{lem:continuous-information-stability}, for every $\eta\in(0,1)$ there exist deterministic constants $K_\eta\geq1$ and $N_\eta\in\mathbb N$ such that
    \begin{equation}
        \label{eq:discrete-continuous-input}
        \mathbb P\left(K_\eta^{-1}\leq J_N\leq K_\eta, \quad \inf_{0\leq t\leq T}C_N(t)\geq K_\eta^{-1} \right)\geq1-\eta, \qquad N\geq N_\eta.
    \end{equation}
Moreover, the moment bound \eqref{eq:discrete-observable-moments} and Markov's inequality control the supremum of $|A_N|$. Combining this with the preceding lower bound on $C_N$, and increasing $K_\eta$ and $N_\eta$ if necessary, we also obtain
    \begin{equation}
        \label{eq:discrete-weight-tightness}
        \mathbb P\left( \sup_{0\leq t\leq T}\left|\frac{A_N(t)}{C_N(t)}\right| \leq K_\eta \right)\geq1-\eta, \qquad N\geq N_\eta.
    \end{equation}
These bounds allow us to estimate the discretization errors on events whose probabilities are arbitrarily close to one.

\begin{lemma}[Stability of the discrete information]
    \label{lem:discrete-information}
    Assume the hypotheses for consistency in either Theorem~\ref{thm:intro-discrete} or Theorem~\ref{thm:degenerate-discrete}. Let $n=n_N$ be a deterministic sequence such that $\Delta_n\to0$.

    Then, for every $\eta\in(0,1)$, there exist deterministic constants $c_\eta>0$ and $N_\eta\in\mathbb N$ such that
        \begin{equation}
            \label{eq:discrete-information-stability}
            \mathbb P\left(J_{N,n}^{\,d}\geq c_\eta\right) \geq1-\eta, \qquad N\geq N_\eta.
        \end{equation}
    Moreover,
        \begin{equation*}
            \frac{\langle L_{N,n}\rangle_T}{J_{N,n}^{\,d}} \xrightarrow{\mathbb P}1, \qquad r_N^2J_{N,n}^{\,d}\xrightarrow{\mathbb P}\infty.
        \end{equation*}
    Under the common-information assumption in the corresponding theorem,
        \begin{equation*}
            J_{N,n}^{\,d}\xrightarrow{\mathbb P}J_\infty,
        \end{equation*}
    where $J_\infty\in(0,\infty)$ is the common deterministic limiting information.
\end{lemma}

\begin{proof}
    By Cauchy-Schwarz and \eqref{eq:discrete-macroscopic-increments}, for $F\in\{A,B,C\}$,
        \begin{equation}
            \label{eq:discrete-integrated-increment-bound}
            \sum_{m=0}^{n-1}\int_{t_m}^{t_{m+1}} \mathbb E|F_N(s)-F_N(t_m)|\,ds \leq K_T(\Delta_n+\delta_N\sqrt{\Delta_n}).
        \end{equation}
    Consider the events
        \begin{equation*}
            G_N(K):=\left\{\inf_{t\leq T}C_N(t)\geq K^{-1},\quad \sup_{t\leq T}|A_N(t)|\leq K\right\}.
        \end{equation*}
    Equations~\eqref{eq:discrete-observable-moments} and \eqref{eq:discrete-continuous-input} imply
        \begin{equation*}
            \lim_{K\to\infty}\limsup_{N\to\infty} \mathbb P(G_N(K)^c)=0.
        \end{equation*}
    On $G_N(K)$, the function $(a,c)\mapsto a^2/c$ is Lipschitz over the range of the observables. Therefore,
        \begin{align*}
            \mathbf1_{G_N(K)}|J_{N,n}^{\,d}-J_N| &\leq K_K\sum_{m=0}^{n-1}\int_{t_m}^{t_{m+1}}\bigl(|A_N(s)-A_N(t_m)|+|C_N(s)-C_N(t_m)|\bigr)ds,\\
            \mathbf1_{G_N(K)}|\mathcal E_{N,n}^{(C)}| &\leq K^4\sum_{m=0}^{n-1}\int_{t_m}^{t_{m+1}} |C_N(s)-C_N(t_m)|\,ds.
        \end{align*}
    Taking expectations, applying \eqref{eq:discrete-integrated-increment-bound}, and then removing the localization proves the two approximation bounds. The remaining assertions follow from \eqref{eq:discrete-continuous-input}, $r_N\to\infty$, and Lemma~\ref{lem:continuous-information-stability}.
\end{proof}

subsection{Consistency}

\begin{proof}[Proof of consistency in Theorems~\ref{thm:intro-discrete} and \ref{thm:degenerate-discrete}]
    Recall from \eqref{eq:discrete-error-ratios} that the estimation error decomposes into
        \begin{equation*}
            \widehat\beta_{N,n}^{\,d}-\beta=R_{N,n}+M_{N,n}.
        \end{equation*}
    We analyze and bound each component separately to show that both converge to zero in probability. We first control the discretization remainder $R_{N,n}$. Write $Z_N=B_N+2\beta A_N$, so that
        \begin{equation*}
            R_{N,n}=\frac{U_{N,n}}{J_{N,n}^{\,d}},\qquad U_{N,n}=\frac12\sum_{m=0}^{n-1}\frac{A_N(t_m)}{C_N(t_m)}\int_{t_m}^{t_{m+1}}\bigl(Z_N(s)-Z_N(t_m)\bigr)\,ds.
        \end{equation*}
    For $K>0$, consider
        \begin{equation*}
            G_N(K)=\left\{\inf_{t\leq T}C_N(t)\geq K^{-1},\quad \sup_{t\leq T}|A_N(t)|\leq K\right\}.
        \end{equation*}
    By \eqref{eq:discrete-observable-moments} and \eqref{eq:discrete-continuous-input},
        \begin{equation*}
            \lim_{K\to\infty}\limsup_{N\to\infty} \mathbb P(G_N(K)^c)=0.
        \end{equation*}
    On this event, $1/C_N(t_m)\leq K$. Taking the absolute value and expectation and applying Cauchy-Schwarz gives
        \begin{align*}
            \mathbb E[\mathbf1_{G_N(K)}|U_{N,n}|] &\leq\frac K2\sum_m\int_{t_m}^{t_{m+1}}\mathbb E\bigl[|A_N(t_m)|\,|Z_N(s)-Z_N(t_m)|\bigr]ds\\
            &\leq\frac K2\sum_m\int_{t_m}^{t_{m+1}} \bigl(\mathbb E[A_N(t_m)^2]\bigr)^{1/2} \bigl(\mathbb E|Z_N(s)-Z_N(t_m)|^2\bigr)^{1/2}ds.
        \end{align*}
    Lemma~\ref{l:macroscopic-regularity} bounds the first factor uniformly. For $\Delta_n\leq1$, its increment estimate and $\delta_N\leq1$ bound the second factor by $K_T\sqrt{s-t_m}$. Consequently,
        \begin{equation*}
            \mathbb E[\mathbf1_{G_N(K)}|U_{N,n}|] \leq K_{K,T}\sum_m\int_{t_m}^{t_{m+1}}\sqrt{s-t_m}\,ds \leq K_{K,T}\sqrt{\Delta_n}.
        \end{equation*}
    Markov's inequality and removal of the localization give $U_{N,n}=O_{\mathbb P}(\sqrt{\Delta_n})$. Since $(J_{N,n}^{\,d})^{-1}=O_{\mathbb P}(1)$, we obtain
        \begin{equation*}
            R_{N,n}\xrightarrow{\mathbb P}0.
        \end{equation*}
    Next, we address the martingale error $M_{N,n}$. The continuous local martingale $L_{N,n}$ defined in \eqref{eq:def-LN-discrete} satisfies
        \begin{equation*}
            \langle L_{N,n}\rangle_T=J_{N,n}^{\,d}+\mathcal E_{N,n}^{(C)}.
        \end{equation*}
    By Lemma~\ref{lem:discrete-information}, $\mathcal E_{N,n}^{(C)}/J_{N,n}^{\,d}\to0$ in probability. Thus \begin{equation*}
        M_{N,n} =\left(\delta_N\frac{L_{N,n}(T)}{\langle L_{N,n}\rangle_T}\right) \frac{\langle L_{N,n}\rangle_T}{J_{N,n}^{\,d}} =:\widetilde M_{N,n} \left(1+\frac{\mathcal E_{N,n}^{(C)}}{J_{N,n}^{\,d}}\right).
    \end{equation*}
    Since the second factor converges to one, it suffices to show that $\widetilde M_{N,n}\to0$ in probability. For any $\lambda\in\mathbb R$, the stochastic exponential
        \begin{equation*}
            \exp\left(\lambda L_{N,n}(t) -\frac{\lambda^2}{2}\langle L_{N,n}\rangle_t\right)
        \end{equation*}
    is a positive local martingale and therefore a supermartingale whose expectation is at most one. Fix $a,\eta>0$ and consider
        \begin{equation*}
            E_{N,n}^{+}:= \left\{L_{N,n}(T)\geq\frac\eta4 r_N\langle L_{N,n}\rangle_T, \quad r_N^2\langle L_{N,n}\rangle_T\geq a\right\}.
        \end{equation*}
    Setting $\lambda=\eta r_N/4$, on $E_{N,n}^{+}$ the exponent is at least $\eta^2a/32$. Markov's inequality therefore gives $\mathbb P(E_{N,n}^{+})\leq\exp(-\eta^2a/32)$. The same argument applied to $-L_{N,n}$ bounds the lower tail. Splitting according to whether $r_N^2\langle L_{N,n}\rangle_T<a$, we deduce
        \begin{equation}
            \label{eq:martingale-tail-bound}
            \mathbb P\left(|\widetilde M_{N,n}|>\frac\eta4\right) \leq\mathbb P\left(r_N^2\langle L_{N,n}\rangle_T<a\right) +2\exp(-\eta^2a/32).
        \end{equation}
    Lemma~\ref{lem:discrete-information} gives $r_N^2J_{N,n}^{\,d}\to\infty$ and $\langle L_{N,n}\rangle_T/J_{N,n}^{\,d}\to1$ in probability. Hence $r_N^2\langle L_{N,n}\rangle_T\to\infty$ in probability. Letting first $N\to\infty$ and then $a\to\infty$ proves $\widetilde M_{N,n}\to0$, and consequently $M_{N,n}\to0$. Since both terms of the estimation error vanish in probability, the estimator is consistent in both regimes.
\end{proof}

\subsection{Asymptotic normality}

\begin{proof}[Proof of asymptotic normality in Theorems~\ref{thm:intro-discrete} and \ref{thm:degenerate-discrete}] 
    By scaling the error decomposition \eqref{eq:discrete-error-ratios} by $r_N$, we write
        \begin{equation}
            \label{eq:discrete-scaled-error}
            r_N(\widehat\beta_{N,n}^{\,d}-\beta) =r_NR_{N,n}+r_NM_{N,n} =\frac{r_NU_{N,n}}{J_{N,n}^{\,d}} +\frac{L_{N,n}(T)}{J_{N,n}^{\,d}}.
        \end{equation}
    The common-information assumption and Lemma~\ref{lem:discrete-information} imply $J_{N,n}^{\,d}\to J_\infty$ in probability. To prove that $r_NR_{N,n}\to0$ in probability, we bound $U_{N,n}$ using the It\^o decomposition of $Z_N=B_N+2\beta A_N$:
        \begin{equation*}
            dZ_N(s)=D_N^Z(s)\,ds+dM_N^Z(s),\qquad D_N^Z=D_N^B+2\beta D_N^A,\quad M_N^Z=M_N^B+2\beta M_N^A.
        \end{equation*}
    Lemma~\ref{l:macroscopic-regularity} gives
        \begin{equation*}
            \sup_{N,t\leq T}\mathbb E|D_N^Z(t)|^2<\infty, \qquad\mathbb E\langle M_N^Z\rangle_T\leq K_T\delta_N^2.
        \end{equation*}
    Substituting this decomposition into $U_{N,n}$ splits it into a drift term $U_{N,n}^{(D)}$ and a martingale term $U_{N,n}^{(M)}$:
        \begin{align*}
            U_{N,n}^{(D)} &=\frac12\sum_m w_{N,m} \int_{t_m}^{t_{m+1}}\int_{t_m}^sD_N^Z(u)\,du\,ds,\\
            U_{N,n}^{(M)} &=\frac12\sum_m w_{N,m} \int_{t_m}^{t_{m+1}}\bigl(M_N^Z(s)-M_N^Z(t_m)\bigr)ds, \qquad w_{N,m}:=\frac{A_N(t_m)}{C_N(t_m)}.
        \end{align*}
    For $K>0$, put
        \begin{equation*}
            H_N(K):=\left\{ \sup_{t\leq T}|A_N(t)/C_N(t)|\leq K\right\}.
        \end{equation*}
    On this event, taking the $L^1$ norm and using the drift bound yields
        \begin{equation*}
            \mathbb E[\mathbf1_{H_N(K)}|U_{N,n}^{(D)}|] \leq\frac K2\sum_m\int_{t_m}^{t_{m+1}} \int_{t_m}^s\mathbb E|D_N^Z(u)|\,du\,ds \leq K_{K,T}\Delta_n.
        \end{equation*}
    By \eqref{eq:discrete-weight-tightness}, removal of the localization gives $U_{N,n}^{(D)}=O_{\mathbb P}(\Delta_n)$. For the martingale term, stochastic Fubini gives
        \begin{equation*}
            \int_{t_m}^{t_{m+1}} \bigl(M_N^Z(s)-M_N^Z(t_m)\bigr)ds =\int_{t_m}^{t_{m+1}}(t_{m+1}-u)\,dM_N^Z(u).
        \end{equation*}
    Use the bounded predictable weights $w_{N,m}^{(K)}=w_{N,m}\mathbf1_{\{|w_{N,m}|\leq K\}}$ and set
        \begin{equation*}
            \Gamma_m^{(K)}:=\frac12w_{N,m}^{(K)} \int_{t_m}^{t_{m+1}}(t_{m+1}-u)\,dM_N^Z(u).
        \end{equation*}
    Each $\Gamma_m^{(K)}$ has conditional mean zero given $\mathcal F_{t_m}^N$, so the cross terms in the second moment vanish. It\^o's isometry gives
        \begin{align*}
            \mathbb E\left[\left(\sum_m\Gamma_m^{(K)}\right)^2\right] &=\sum_m\mathbb E[(\Gamma_m^{(K)})^2]\\
            &\leq\frac{K^2}{4}\Delta_n^2 \mathbb E\langle M_N^Z\rangle_T \leq K_{K,T}\delta_N^2\Delta_n^2.
        \end{align*}
    The truncated expression equals $U_{N,n}^{(M)}$ on $H_N(K)$. Removing the truncation therefore yields
        \begin{equation}
            U_{N,n}^{(D)} = O_{\mathbb P}(\Delta_n),\qquad U_{N,n}^{(M)}=O_{\mathbb P}(\delta_N\Delta_n),\qquad U_{N,n}=O_{\mathbb P}(\Delta_n).
        \end{equation}
    Consequently,
        \begin{equation*}
            r_NR_{N,n}=O_{\mathbb P}(r_N\Delta_n) \xrightarrow{\mathbb P}0
        \end{equation*}
    under the mesh condition in the corresponding theorem.

    We establish the asymptotic normality of the martingale numerator $L_{N,n}(T)$ via the central limit theorem for martingale difference arrays. The assumptions give tightness of the weights, but no inverse-moment bound for $C_N$, so we use predictable truncation. Set
        \begin{equation*}
            K_N  =\Delta_n^{-1/8},\qquad
            w_{N,m}^{*} = w_{N,m}\mathbf1_{\{|w_{N,m}|\leq K_N\}},\qquad
            \xi_{m,N} = w_{N,m}^{*} \int_{t_m}^{t_{m+1}}\sqrt{C_N(s)}\,dW_s.
        \end{equation*}
    Since $K_N\to\infty$ and the supremum of the weights is tight,
        \begin{equation}
            \label{eq:discrete-array-localization}
            \mathbb P\left(\sum_{m=0}^{n-1}\xi_{m,N}\neq L_{N,n}(T)\right)\longrightarrow0.
        \end{equation}
    Write $\mathbb E_m[\cdot]=\mathbb E[\cdot\mid\mathcal F_{t_m}^N]$. We verify the following three properties of the array.

\medskip

    \noindent\textit{(i) Vanishing conditional expectations.}
    The prefactor is bounded and $\mathcal F_{t_m}^N$-measurable. Since $\sup_{N,t\leq T}\mathbb E C_N(t)<\infty$, the stochastic integral is square-integrable and
        \begin{equation*}
            \mathbb E_m[\xi_{m,N}] = 0.
        \end{equation*}

\medskip

    \noindent\textit{(ii) Convergence of the conditional variances to $J_\infty$.} Applying It\^o's isometry conditionally gives
        \begin{equation*}
            \mathbb E_m[\xi_{m,N}^2] =(w_{N,m}^{*})^2 \int_{t_m}^{t_{m+1}}\mathbb E_m[C_N(s)]\,ds.
        \end{equation*}
    Adding and subtracting $C_N(t_m)$, we decompose the sum as
        \begin{equation*}
            \sum_m\mathbb E_m[\xi_{m,N}^2] = \Delta_n\sum_m(w_{N,m}^{*})^2C_N(t_m)+D_{N,n}^{*},
        \end{equation*}
    where
        \begin{equation*}
            D_{N,n}^{*}:=\sum_m(w_{N,m}^{*})^2 \int_{t_m}^{t_{m+1}}\mathbb E_m[C_N(s)-C_N(t_m)]\,ds.
        \end{equation*}
    The first term equals $J_{N,n}^{\,d}$ whenever no weight is truncated, and hence converges in probability to $J_\infty$. For the remainder, conditional Jensen's inequality and \eqref{eq:discrete-integrated-increment-bound} give
        \begin{align*}
            \mathbb E|D_{N,n}^{*}| &\leq K_N^2\sum_m\int_{t_m}^{t_{m+1}} \mathbb E|C_N(s)-C_N(t_m)|\,ds\\
            &\leq K_TK_N^2(\Delta_n+\delta_N\sqrt{\Delta_n}) = K_T(\Delta_n^{3/4}+\delta_N\Delta_n^{1/4}) \longrightarrow0.
        \end{align*}
    Thus $D_{N,n}^{*}\to0$ in $L^1$, and
        \begin{equation*}
            \sum_m\mathbb E_m[\xi_{m,N}^2] \xrightarrow{\mathbb P}J_\infty.
        \end{equation*}

\medskip

    \noindent\textit{(iii) Lyapunov condition: vanishing fourth conditional moments.}
    The conditional Burkholder-Davis-Gundy inequality and Cauchy-Schwarz yield
        \begin{align*}
            \mathbb E_m[\xi_{m,N}^4] &\leq K(w_{N,m}^{*})^4 \mathbb E_m\left[\left(\int_{t_m}^{t_{m+1}}C_N(s)\,ds\right)^2\right]\\
            &\leq KK_N^4\Delta_n  \int_{t_m}^{t_{m+1}}\mathbb E_m[C_N(s)^2]\,ds.
        \end{align*}
    Taking the unconditional expectation of the sum and applying \eqref{eq:discrete-observable-moments}, we obtain
        \begin{equation*}
            \mathbb E\left[\sum_m\mathbb E_m[\xi_{m,N}^4]\right] \leq K_TK_N^4\Delta_n = K_T\sqrt{\Delta_n}\longrightarrow0.
        \end{equation*}  
    In particular, for every $a>0$,
        \begin{equation*}
            \sum_m\mathbb E_m[\xi_{m,N}^2\mathbf1_{\{|\xi_{m,N}|>a\}}] \leq a^{-2}\sum_m\mathbb E_m[\xi_{m,N}^4] \xrightarrow{\mathbb P}0.
        \end{equation*}
    This verifies the conditional Lindeberg condition.

    The array is conditionally centered, its conditional variance converges to $J_\infty$, and it satisfies the conditional Lindeberg condition. The deterministic-variance martingale difference array central limit theorem therefore gives
        \begin{equation*}
            \sum_m\xi_{m,N}\xrightarrow{\mathcal L} \mathcal N(0,J_\infty);
        \end{equation*}
    see \cite[Chapter~3]{HalHey80}. By \eqref{eq:discrete-array-localization}, the same limit holds for $L_{N,n}(T)$. Since $J_{N,n}^{\,d}\to J_\infty$ in probability, Slutsky's theorem yields
        \begin{equation*}
            r_NM_{N,n}=\frac{L_{N,n}(T)}{J_{N,n}^{\,d}} \xrightarrow{\mathcal L}\mathcal N(0,J_\infty^{-1}).
        \end{equation*}
    Finally, $r_NR_{N,n}\to0$ in probability, so a second application of Slutsky's theorem completes the proof:
        \begin{equation*}
            r_N(\widehat\beta_{N,n}^{\,d}-\beta) \xrightarrow{\mathcal L}\mathcal N(0,J_\infty^{-1}).
        \end{equation*}
    Here $r_N=\sqrt{N/\varepsilon_N}$ in the strictly positive regime, whereas $r_N=N$ in the degenerate regime. The respective mesh conditions are therefore \eqref{eq:intro-discrete-mesh-condition} and \eqref{eq:degenerate-discrete-mesh}.
\end{proof}

\subsection{Proof of the macroscopic regularity lemma}

\begin{proof}[Proof of Lemma~\ref{l:macroscopic-regularity}]
    Let $E$ be the invariant interval from Condition~\ref{ass:discrete-regularity}. We apply It\^o's formula to the $C^2$ extensions of the observable functions and use the same notation for these extensions. All coefficient evaluations and interpolating segments below lie in $E$ or $E^2$, so only the derivative bounds on that interval are used.

    We use the empirical moment bounds from Lemma~\ref{lem:continuous-observable-convergence}; no exchangeability of the ordered coordinates is required. In the notation $m_{p,N}(t):=\langle\mu_t^N,|x|^p\rangle$, these give
        \begin{equation*}
            \sup_N\mathbb E\sup_{t\leq T}m_{8,N}(t)<\infty.
        \end{equation*}

    \medskip
    
    \noindent\textit{Uniform moments of the observables.} The coefficient bounds imply
        \begin{equation*}
            |A_N(t)|+|B_N(t)|\leq K(1+m_{2,N}(t)), \qquad |C_N(t)|\leq K(1+m_{4,N}(t)).
        \end{equation*}
    Jensen's inequality gives $m_{2,N}^4\leq m_{8,N}$ and $m_{4,N}^2\leq m_{8,N}$. Taking suprema and expectations proves \eqref{eq:discrete-observable-moments}.

    \medskip
    
    \noindent\textit{The one-particle observables $B_N$ and $C_N$.} For $f=\phi_{B,N}$ or $f=\phi_C$, apply It\^o's formula to $N^{-1}\sum_i f(X_t^{i,N})$. Writing $x_i=X_t^{i,N}$ in the drift expression, symmetrization gives
        \begin{align*}
            D_N^f(t) &=\frac1N\sum_i f'(x_i)b(x_i) +\frac{\varepsilon_N}{2N}\sum_i f''(x_i)\sigma^2(x_i) + \frac{\beta}{2N^2}\sum_{i\neq j}H(x_i,x_j)\mathcal D_f(x_i,x_j), \\
            M_N^f(t) &=\frac{\sqrt{\varepsilon_N}}{N}\sum_i \int_0^t f'(X_s^{i,N})\sigma(X_s^{i,N})\,dB_s^i.
        \end{align*}
    Here $\mathcal D_f$ is the divided difference defined in Section~\ref{s:continuous}. The derivative bounds in Condition~\ref{ass:discrete-regularity} imply
        \begin{equation*}
            |\mathcal D_{\phi_{B,N}}(x,y)|\leq K, \qquad |\mathcal D_{\phi_C}(x,y)|\leq K(1+x^2+y^2).
        \end{equation*}
    Combining these bounds with Condition~\ref{ass:bounds} and Young's inequality yields
        \begin{equation*}
            |D_N^B(t)|\leq K(1+m_{2,N}(t)), \qquad |D_N^C(t)|\leq K(1+m_{4,N}(t)).
        \end{equation*}
    Their squared expectations are uniformly bounded. Moreover,
        \begin{equation*}
            \frac{d}{dt}\langle M_N^f\rangle_t = \frac{\varepsilon_N}{N^2}\sum_i \bigl(f'(x_i)\sigma(x_i)\bigr)^2.
        \end{equation*}
    The integrand has growth at most degree four for $f=\phi_{B,N}$ and degree eight for $f=\phi_C$. Therefore,
        \begin{equation*}
            \frac{d}{dt}\langle M_N^B\rangle_t \leq K\delta_N^2(1+m_{4,N}(t)), \qquad \frac{d}{dt}\langle M_N^C\rangle_t \leq K\delta_N^2(1+m_{8,N}(t)).
        \end{equation*}

    \medskip
    
    \noindent\textit{The two-particle observable $A_N$.} Applying multidimensional It\^o's formula to $A_N=(2N^2)^{-1}\sum_{i\neq j}H(X^{i,N},X^{j,N})$ gives $dA_N=D_N^A\,dt+dM_N^A$. Its regular drift is
        \begin{equation*}
            D_{N,\mathrm{reg}}^A(t) = \frac1{N^2}\sum_{i\neq j}\partial_1H(x_i,x_j)b(x_i) + \frac{\varepsilon_N}{2N^2}\sum_{i\neq j} \partial_{11}H(x_i,x_j)\sigma^2(x_i).
        \end{equation*}
    The symmetry of $H$ combines the two singular drift sums into
        \begin{equation*}
            D_{N,\mathrm{sing}}^A(t) = \frac{\beta}{N^3}\sum_{i\neq j}\sum_{k\neq i} \partial_1H(x_i,x_j)\frac{H(x_i,x_k)}{x_i-x_k}.
        \end{equation*}
    Split the sum into $S_1$, corresponding to $j=k$, and $S_2$, corresponding to pairwise distinct indices. Swapping $i,k$ and averaging gives exactly
        \begin{align}
            S_1 &=\frac{\beta}{2N^3}\sum_{i\neq k}H(x_i,x_k) \frac{\partial_1H(x_i,x_k)-\partial_2H(x_i,x_k)}{x_i-x_k},\\
            S_2 &=\frac{\beta}{2N^3} \sum_{\substack{i,j,k\\\text{pairwise distinct}}}H(x_i,x_k) \frac{\partial_1H(x_i,x_j)-\partial_1H(x_k,x_j)}{x_i-x_k}.
        \end{align}
    Both singularities are removable. For the first quotient, symmetry gives $\partial_1H(y,y)=\partial_2H(y,y)$ and hence
        \begin{equation*}
            \frac{\partial_1H(x,y)-\partial_2H(x,y)}{x-y} = \int_0^1(\partial_{11}H-\partial_{12}H) (y+\theta(x-y),y)\,d\theta.
        \end{equation*}
    For the second quotient,
        \begin{equation*}
            \frac{\partial_1H(x,y)-\partial_1H(z,y)}{x-z} = \int_0^1\partial_{11}H(z+\theta(x-z),y)\,d\theta.
        \end{equation*}
    These formulas also define the continuous extensions to the diagonals. By Condition~\ref{ass:discrete-regularity}, their absolute values are bounded by $K(1+x^2+y^2)$ and $K(1+x^2+y^2+z^2)$, respectively. After multiplication by $H$ in $S_1,S_2$, every bounding monomial has total degree at most four. Young's inequality and the factors $N^{-2}$ and $N^{-3}$ in the empirical sums consequently give
        \begin{equation*}
            |D_N^A(t)|\leq K(1+m_{4,N}(t)).
        \end{equation*}
    Since $m_{4,N}^2\leq m_{8,N}$, this proves the required squared-drift bound for $A_N$.

    Grouping the stochastic integrals by their driving Brownian motions gives
        \begin{equation*}
            M_N^A(t)=\frac{\sqrt{\varepsilon_N}}{N^2}\sum_i \int_0^t\sigma(X_s^{i,N}) \left(\sum_{j\neq i}\partial_1H(X_s^{i,N},X_s^{j,N})\right)dB_s^i.
        \end{equation*}
    Therefore,
        \begin{align*}
            \frac{d}{dt}\langle M_N^A\rangle_t &=\frac{\varepsilon_N}{N^4}\sum_i\sigma^2(x_i)\left(\sum_{j\neq i}\partial_1H(x_i,x_j)\right)^2\\
                &\leq\frac{\varepsilon_N}{N^3}\sum_{i\neq j} \sigma^2(x_i)|\partial_1H(x_i,x_j)|^2 \leq K\delta_N^2(1+m_{6,N}(t)).
        \end{align*}
    This proves \eqref{eq:discrete-drift-bound} and \eqref{eq:discrete-macroscopic-bracket-bound} for all three observables. In particular their local martingale parts are square-integrable. Finally, for each $F\in\{A,B,C\}$, Cauchy-Schwarz and It\^o's isometry give
        \begin{align*}
            \mathbb E|F_N(t)-F_N(s)|^2 &\leq2(t-s)\int_s^t\mathbb E|D_N^F(u)|^2\,du +2\mathbb E\bigl[\langle M_N^F\rangle_t-\langle M_N^F\rangle_s \bigr]\\
            &\leq K_T\bigl((t-s)^2+\delta_N^2(t-s)\bigr).
        \end{align*}
    Summation over $F$ proves \eqref{eq:discrete-macroscopic-increments}. All It\^o identities are first justified by localization. In the non-sticky regime the symmetrizations hold for Lebesgue-a.e. time, which suffices for the integrated identities; their continuous divided-difference extensions justify all bounds at collision configurations as well.
\end{proof}

\bibliographystyle{plain}
\bibliography{references}

\end{document}